\documentclass[11pt, reqno]{amsart}%
\usepackage{amsmath, amstext, amsbsy, amssymb, amscd}
\usepackage{amsmath}
\usepackage{amsxtra}
\usepackage{amscd}
\usepackage{amsthm}
\usepackage{amsfonts}
\usepackage{amssymb}
\usepackage{eucal}
\usepackage{color}
\usepackage{graphicx}%
\newtheorem{theorem}{Theorem}[section]
\theoremstyle{plain}

\newtheorem{lemma}[theorem]{Lemma}
\newtheorem{proposition}[theorem]{Proposition}

\theoremstyle{definition}
\newtheorem{definition}[theorem]{Definition}

\theoremstyle{remark}
\newtheorem{remark}[theorem]{Remark}

\definecolor{A}{rgb}{.75,1,.75}

\numberwithin{equation}{section}

\newcommand{\C}{\mathbb C}
\newcommand{\Z}{\mathbb Z}
\newcommand{\Cl}{\mathcal C}
\newcommand{\h}{\mathfrak{h}}

\newcommand{\be}{\beta}

\newcommand{\al}{\alpha}

\newcommand{\wtd}{\widetilde}
\newcommand{\td}{\tilde}

\newcommand{\aH}{\mathfrak{H}}
\newcommand{\cdaH}{\ddot{\mathfrak{H}}^{\sim}} 
\newcommand{\daha}{\ddot{\mathfrak{H}}} 
\newcommand{\dahc}{\ddot{\mathfrak H}^{\mathfrak c}} 
\newcommand{\sdaha}{\ddot{\mathfrak H}^-} 

{\vskip-\lastskip\medskip
  \noindent
  {\em #1.}\enspace
  }%
{\qed\par\medskip
  }

\begin{document}

\title[Rational spin double affine Hecke algebras]
{Hecke-Clifford algebras and spin Hecke algebras II:\\ the
rational double affine type}

\author[Ta Khongsap]{Ta Khongsap}

\author[Weiqiang Wang]{Weiqiang Wang}
\address{Department of Math., University of Virginia,
Charlottesville, VA 22904} \email{tk7p@virginia.edu (Khongsap);
ww9c@virginia.edu (Wang)}

\begin{abstract}
The notion of rational spin double affine Hecke algebras (sDaHa)
and rational double affine Hecke-Clifford algebras (DaHCa)
associated to classical Weyl groups are introduced. The basic
properties of these algebras such as the PBW basis and Dunkl
operator representations are established. An algebra isomorphism
relating the rational DaHCa to the rational sDaHa is obtained. We
further develop a link between the usual rational Cherednik
algebra and the rational sDaHa by introducing a notion of rational
covering double affine Hecke algebras.
\end{abstract}

\maketitle

\setcounter{tocdepth}{1} \tableofcontents

\section{Introduction}

\subsection{}

The rational Cherednik algebra (see Etingof-Ginzburg \cite{EG} and
also Drinfeld \cite{Dr} for a more general deformation
construction) is a degenerate version of the double affine Hecke
algebra \cite{Ch}, and it admits a polynomial representation via
the Dunkl operators \cite{Dun}. A similar degeneration in the case
of affine Hecke algebras was introduced and studied earlier in
\cite{Dr} in the type $A$ case and by Lusztig in general
\cite{Lu1, Lu2}. The rational Cherednik algebra has a rich
representation theory and it affords various interesting
connections to integrable systems, noncommutative geometry, etc.
We refer to Etingof \cite{Et} and Rouquier \cite{Rou} for reviews
and extensive references. The rational Cherednik algebra with one
particular parameter being zero, denoted by $\daha_W$ in this
Introduction, is known to have a large center \cite{EG} (cf.
Gordon \cite{Gor}).

In \cite{W1}, the second author introduced (degenerate) spin Hecke
algebras of affine and double affine type as well as double affine
Hecke-Clifford algebra, associated to I. Schur's spin symmetric
group \cite{Sch}. These algebras were shown to be closely related
to the affine Hecke-Clifford algebra of Nazarov \cite{Naz}. The
spin affine Hecke algebras and affine Hecke-Clifford algebras
associated to all classical Weyl groups have been recently
constructed by the authors \cite{KW}.

\subsection{}

In this paper we shall construct three classes of closely related
(super) algebras associated to each {\em classical} finite Weyl
group $W$: the rational double affine Hecke-Clifford algebra
(DaHCa) $\dahc_W$, the rational spin double affine Hecke algebra
(sDaHa) $\sdaha_W$, and the rational covering double affine Hecke
algebra (cDaHa) $\cdaH_W$. We show that the algebras $\dahc_W$ and
$\sdaha_W$ are Morita super-equivalent (in the terminology of
\cite{W2}) and that $\cdaH_W$ has both the rational Cherednik
algebra $\daha_W$ and the sDaHa $\sdaha_W$ as its natural
quotients. Some basic properties including the PBW basis theorem
and Dunkl operator realizations of these algebras are further
established.

We expect that these algebras afford very interesting
representation theory and connections with noncommutative
geometry.

\subsection{}

In \cite{Mo, KW} a double cover $\wtd{W}$ of the finite Weyl group
$W$ associated to a distinguished $2$-cocycle
\begin{eqnarray*} \label{ext}
1 \longrightarrow \Z_2 \longrightarrow \wtd{W} \longrightarrow W
\longrightarrow 1
\end{eqnarray*}
was considered. Denote $\Z_2 =\{1,z\}$. From now on, let $W$ be
one of the classical Weyl groups. In this paper, we define the
algebras $\dahc_W$, $\sdaha_W$ and $\cdaH_W$ for every $W$ of type
$A_{n-1}, D_n, B_n$ in terms of explicit generators and relations,
where the number of parameters in each of these algebras is the
number of conjugacy classes of reflections in $W$. The
compatibility among the defining relations for $\dahc_W$ (which would
imply the PBW basis property and similar compatibility for
$\sdaha_W$ and $\cdaH_W$ when combined with other results)
requires lengthy but elementary case-by-case verifications. In a
suitable sense, the defining relations are naturally and uniquely
dictated by the compatibility of these relations.

As is well known, the rational Cherednik algebra $\daha_W$ has a
triangular decomposition with the group algebra $\C W$ as its
middle term. We show that the algebras $\dahc_W$, $\sdaha_W$ and
$\cdaH_W$ also afford triangular decompositions which contain
$\Cl_n \rtimes \C W$, $\C W^-$ and $\C \widetilde{W}$ respectively
as the middle terms, where $\Cl_n$ denotes the Clifford algebra of
the reflection representation of $W$ with a natural $W$-action.
For instance, the rational DaHCa $\dahc_W$ and sDaHa $\sdaha_W$
have the following triangular decompositions:
\begin{eqnarray*}
\dahc_W & \cong & \C [x_1,\ldots,x_n] \otimes(\Cl_n \rtimes
\mathbb{C }W)
\otimes\C [y_1,\ldots,y_n] \\
\sdaha_W & \cong & \Cl[\xi_1,\ldots,\xi_n] \otimes\mathbb{C }W^-
\otimes\C [y_1,\ldots,y_n]
\end{eqnarray*}
where $\Cl[\xi_1,\ldots,\xi_n]$ is a noncommutative algebra with
$\xi_i \xi_j = -\xi_j \xi_i$ for $i\neq j$. The relations between
$\C W^-$ and $\Cl[\xi_1,\ldots,\xi_n]$ involve subtle signs
similar to those appearing in the spin affine Hecke algebras
defined in \cite{KW, W1}.

We further show that the algebras $\dahc_W$ and $\sdaha_W$ have
large centers which contain $\C [y_1,\ldots,y_n]^W$ and $\C
[x_1^2,\ldots, x_n^2]^W$ (and $\C [\xi_1^2,\ldots,\xi_n^2]^W$
respectively) as subalgebras. In particular, the algebras
$\dahc_W$ and $\sdaha_W$ are module-finite over their centers.

The group algebra $\C W$ and the spin Weyl group algebra $\C W^-$
appear as natural quotients of $\C \widetilde{W}$ by the ideals
$\langle z\mp 1\rangle$ respectively. We show that these quotient
maps, denoted by $\Upsilon_\pm$, extend to the setup of double
affine Hecke algebras. All these statements can be summarized in
the following commutative diagram with the vertical arrows being
natural inclusions:
$$
\begin{CD}
 \C W @<\Upsilon_+<< \C \widetilde{W} @>\Upsilon_->> \C W^- \\
 @VVV @VVV @VVV \\
 \daha_W @<\Upsilon_+<< \cdaH_W @>\Upsilon_->> \sdaha_W
  \end{CD}
$$

In \cite{KW}, we established a superalgebra isomorphism $\Phi:
\Cl_n \rtimes \C W \stackrel{\simeq}{\longrightarrow} \Cl_n
\otimes \C W^-$ (which actually holds also for exceptional Weyl
groups), generalizing the type $A$ result of Sergeev and
Yamaguchi. In this paper, we shall establish a Morita
super-equivalence between $\dahc_W$ and $\sdaha_W$, or more
explicitly, a superalgebra isomorphism between $\dahc_W$ and the
tensor algebra $\Cl_n \otimes \sdaha_W$ which extends the
isomorphism $\Phi$ (see \cite{W1} for the type $A$ case). This can
be summarized conveniently in the following commutative diagram
with the vertical arrows being natural inclusions:
$$
\begin{CD}
 \Cl_n \rtimes \C W @>\cong>\Phi> \Cl_n \otimes \C W^- \\
 @VVV @VVV \\
 \dahc_W @>\cong>\Phi> \Cl_n \otimes \sdaha_W
  \end{CD}
$$

\subsection{}

As our constructions in a way rely on a choice of orthonormal
basis of $\mathfrak h$, they do not seem to be easily extendable
to the exceptional Weyl groups. Also, in contrast to the setup of
rational Cherednik algebras in \cite{EG}, our Hecke algebras do
not seem to afford an extra parameter in a  natural way to
trivialize their center.

There has been another attempt (cf. Chmutova \cite{Chm}) to
generalize the rational Cherednik algebras and more generally
symplectic reflection algebras by adding a twist with a
$2$-cocycle of a finite group. But the approach therein does not
produce intrinsically interesting new algebras with nontrivial
$2$-cocycles of the Weyl groups as in our approach.

\subsection{}

The paper is organized as follows. In Section~\ref{sec:finite}, we
recall some facts about the distinguished double covers of the
Weyl groups. For more detailed treatment, consult \cite{KW}. We
introduce in Section~\ref{sec:daHCa} the rational DaHCa $\dahc_W$
and establish its PBW basis property. In
Section~\ref{sec:daHCaDunkl} the Dunkl operator representations of
$\dahc_W$ are obtained. Section~\ref{sec:sdaha} and
\ref{sec:sdahaDunkl} are the counterparts for the sDaHa $\sdaha_W$
of Section~\ref{sec:daHCa} and \ref{sec:daHCaDunkl} respectively.
In addition, the superalgebra isomorphism $\Phi$ relating the
sDaHa and DaHCa is established in Section~\ref{sec:sdaha}.
Finally, in Section~\ref{sec:coveringHecke} the rational cDaHa
$\cdaH_W$ is introduced and it provides a link between the sDaHa
$\sdaha_W$ and the usual DaHa $\daha_W$. Finally, in the Appendix
(Section~\ref{sec:Appendix}), we present the proofs of several
lemmas in Section~\ref{sec:daHCa} and \ref{sec:daHCaDunkl}.

{\bf Acknowledgements.} W.W. is partially supported by an NSF
grant.

\section{The spin Weyl groups} \label{sec:finite}

In this section, we recall from \cite{KW} some preliminary setups
which lead to Theorem~\ref{th:isofinite} below, but here we will
restrict ourselves to classical Weyl groups only. This is all we
need in the subsequent sections.
\subsection{A double covering of Weyl groups}
\label{subsec:spinWeyl}

Let $W$ be an (irreducible) finite Weyl group of classical type
(i.e. of type $A,B,D$) with the following presentation:
\begin{eqnarray} \label{eq:weyl}
\langle s_1,\ldots,s_n | (s_is_j)^{m_{ij}} = 1,\ m_{i i} = 1,
 \ m_{i j} = m_{j i} \in \Z_{\geq 2}, \text{for } i
 \neq j \rangle.
\end{eqnarray}
The integers $m_{i j}$ are specified by the Coxeter-Dynkin diagrams
whose vertices correspond to the generators of $W$ below. By
convention, we only mark the edge connecting $i,j$ with $m_{ij} \ge
4$. We have $m_{ij}=3$ for $i \neq j$ connected by an unmarked edge,
and $m_{ij}=2$ if $i,j$ are not connected by an edge.

 \begin{equation*}
 \begin{picture}(150,45) 
 \put(-99,18){$A_{n}$}
 \put(-30,20){$\circ$}
 \put(-23,23){\line(1,0){32}}
 \put(10,20){$\circ$}
 \put(17,23){\line(1,0){23}}
 \put(41,22){ \dots }
 \put(64,23){\line(1,0){18}}
 \put(82,20){$\circ$}
 \put(89,23){\line(1,0){32}}
 \put(122,20){$\circ$}
 \put(-30,9){$1$}
 \put(10,9){$2$}
 \put(74,9){${n-1}$}
 \put(122,9){${n}$}
 \end{picture}
 \end{equation*}
 %
 \begin{equation*}
 \begin{picture}(150,55) 
 \put(-99,18){$B_{n}(n\ge 2)$}
 \put(-30,20){$\circ$}
 \put(-23,23){\line(1,0){32}}
 \put(10,20){$\circ$}
 \put(17,23){\line(1,0){23}}
 \put(41,22){ \dots }
 \put(64,23){\line(1,0){18}}
 \put(82,20){$\circ$}
 \put(89,23){\line(1,0){32}}
 \put(122,20){$\circ$}
 \put(-30,10){$1$}
 \put(10,10){$2$}
 \put(74,10){${n-1}$}
 \put(122,10){${n}$}
 %
 \put(102,24){$4$}
 \end{picture}
 \end{equation*}
%
%
 \begin{equation*}
 \begin{picture}(150,75) 
 \put(-99,28){$D_{n} (n \ge 4)$}
 \put(-30,30){$\circ$}
 \put(-23,33){\line(1,0){32}}
 \put(10,30){$\circ$}
 \put(17,33){\line(1,0){15}}
 \put(35,30){$\cdots$}
 \put(52,33){\line(1,0){15}}
 \put(68,30){$\circ$ }
 \put(75,33){\line(1,0){32}}
 \put(108,30){$\circ$}
 \put(113,36){\line(1,1){25}}
 \put(138,61){$\circ$}
 \put(113,29){\line(1,-1){25}}
 \put(138,-1){$\circ$}
 \put(-29,20){$1$}
 \put(10,20){$2$}
 \put(60,20){$n-3$}
 \put(117,30){$n-2$}
 \put(145,0){$n-1$}
 \put(145,60){$n$}
 %
 \end{picture}
 \end{equation*}

\vspace{.5cm}

We shall be concerned about a distinguished double covering
$\wtd{W}$ of $W$:
$$
1 \longrightarrow \Z_2 \longrightarrow \wtd{W} \longrightarrow W
\longrightarrow 1.
$$
We denote by $\Z_2 =\{1,z\},$ and by $\td{t}_i$ a fixed preimage of
the generators $s_i$ of $W$ for each $i$. The group $\wtd{W}$ is
generated by $z, \td{t}_1,\ldots, \td{t}_n$ with relations (besides
the obvious relation that $z$ is central of order $2$) listed in the
following table, which corresponds to setting the $\alpha_i$ for all
$i$ in Karpilovsky \cite[Table 7.1]{Kar} to be $z$.

\begin{center}
\vskip 1pc
\vbox{\tabskip=0pt \offinterlineskip
\def\tablerule{\noalign{\hrule}}
\halign to345pt {\strut#& \vrule# \tabskip=1em plus2em& \hfil#&
\vrule#& #\hfil& \vrule#\tabskip=0pt \cr\tablerule
&&\omit\hidewidth $W$\hidewidth&& \omit\hidewidth
 Defining Relations for $\wtd{W}$ \hidewidth&\cr\tablerule
&& &&$\td{t}_i^2 =1, 1 \le i \le n,$&\cr
&&$A_n \quad$ &&$ \td{t}_i \td{t}_{i+1} \td{t}_i = \td{t}_{i+1}
\td{t}_i \td{t}_{i+1}, 1 \le i \le n-1$&\cr
&& &&$ \td{t}_i \td{t}_j = z \td{t}_j \td{t}_i \text{ if } m_{ij}
=2$&\cr\tablerule
&& &&$\td{t}_i^2 =1, 1 \le i \le n, \;\;\; \td{t}_i \td{t}_{i+1}
\td{t}_i= \td{t}_{i+1} \td{t}_i \td{t}_{i+1}, 1 \le i \le n-2$&\cr
&&$B_n \quad$ &&$\td{t}_i \td{t}_j = z \td{t}_j \td{t}_i, 1\le i<j
\le n-1, m_{ij} =2$&\cr
&& $ (n \ge 2)$ &&$\td{t}_i \td{t}_n =z \td{t}_n \td{t}_i, 1\le i
\le n-2$&\cr
&& &&$(\td{t}_{n-1} \td{t}_n)^2 = z (\td{t}_n
\td{t}_{n-1})^2$&\cr\tablerule
%
&& &&$\td{t}_i^2 =1, 1 \le i \le n$&\cr
&&$D_n \quad$ &&$\td{t}_i \td{t}_j \td{t}_i
= \td{t}_j \td{t}_i \td{t}_j \text{ if } m_{ij} =3$&\cr
&&$ (n \ge 4)$ &&$\td{t}_i \td{t}_j =z \td{t}_j \td{t}_i, 1 \le i<j
\le n, m_{ij} =2$&\cr\tablerule 
}

\vspace{.1cm}
 TABLE 1: The defining relations of $\widetilde{W}$
 }
\end{center}

\vskip 2pc

The quotient algebra $\C W^- :=\C \wtd{W} /\langle z+1\rangle$ of
$\C \wtd{W}$ by the ideal generated by $z+1$ will be called the {\em
spin Weyl group algebra} associated to $W$. Denote by $t_i \in \C
W^-$ the image of $\td{t}_i$.
%
The spin Weyl group algebra $\C W^-$ has a natural superalgebra
(i.e. $\Z_2$-graded algebra) structure by letting each $t_i$ be odd.
The algebra $\C W^-$ is generated by $t_1,\ldots, t_n$ with the
labeling as in the Coxeter-Dynkin diagrams and the explicit
relations summarized in the following table.
\bigskip%

 \begin{center}
\begin{tabular}
[t]{|l|l|}\hline Type of $W$ & Defining Relations for $\C W^-$\\
\hline $\qquad A_{n}$ & $t_{i}^{2}=1$,
$t_{i}t_{i+1}t_{i}=t_{i+1}t_{i}t_{i+1}$,\\ &
$(t_{i}t_{j})^{2}=-1\text{ if } |i-j|\,>1$\\
\hline & $t_{1},\ldots,t_{n-1}$ satisfy the relations for $\C W^-_{A_{n-1}}$, \\
$\qquad B_{n}$ & $t_{n}^{2}=1,(t_{i}t_{n})^{2}=-1$ if $i\neq n-1,n$, \\
&
$(t_{n-1}t_{n})^{4}=-1$\\
\hline & $t_{1},\ldots,t_{n-1}$ satisfy the relations for
$\C W^-_{A_{n-1}}$,\\
$\qquad D_{n}$& $t_{n}^{2}=1,(t_{i}t_{n})^{2}=-1$ if $i\neq n-2,
n$,
\\& $t_{n-2}t_{n}t_{n-2}=t_{n}t_{n-2}t_{n}$\\\hline
\end{tabular}

\vspace{.1cm}
 {TABLE 2: The defining relations of $\C W^-$}
 \end{center}

\bigskip%

By definition, the quotient by the ideal $\langle z-1\rangle$ of the
group algebra $\C \wtd{W}$ is isomorphic to $\C W$.

\bigskip%

\subsection{A superalgebra isomorphism} \label{subsec:iso}

Denote by $\h =\C^n$ the natural representation (respectively the
reflection representation) of the Weyl group $W$ of type $A_{n-1}$
(respectively of type $B_n,$ and $D_n$). Note that $\h$ carries a
$W$-invariant nondegenerate bilinear form $(-,-)$, which gives
rise to an identification $\h^*\cong \h$ and also a bilinear form
on $\h^*$ which will be again denoted by $(-,-)$.

Denote by $\Cl_n$ the Clifford algebra associated to $(\h,
(-,-))$. We shall denote by $\{c_i\}$ the generators in $\Cl_n$
corresponding to a standard orthonormal basis $\{e_i\}$ of $\C^n$
and denote by $\{\be_i\}$ the elements of $\Cl_n$ corresponding to
the simple roots $\{\al_i\}$ normalized with
$$
\be_i^2=1.
$$
More explicitly, $\mathcal{C}_n$ is generated by
$c_{1},\ldots,c_n$ subject to the relations
\begin{align} \label{clifford}
c_{i}^{2} =1,\qquad c_{i}c_{j} =-c_{j}c_{i}\quad (i\neq j).
\end{align}
For type $A_{n-1}$, we have
$\be_{i}=\frac{1}{\sqrt{2}}(c_{i}-c_{i+1}),1\leq i\leq n-1$. For
type $B_{n}$, we have an additional $\be_{n}=c_{n}$, and for type
$D_n$, $\be_{n}=\frac{1}{\sqrt{2}}(c_{n-1}+c_{n})$.

The action of $W$ on $\h$ and $\h^*$ preserves the bilinear form
$(-,-)$ and thus it acts as automorphisms of the algebra $\Cl_n$.
This gives rise to a semi-direct product $\Cl_n \rtimes \C W$.
Moreover, the algebra $\Cl_n \rtimes \C W$ naturally inherits the
superalgebra structure by letting elements in $W$ be even and each
$c_i$ be odd.

Given two superalgebras $\mathcal{A}$ and $\mathcal{B}$, we view
the tensor product of superalgebras $\mathcal{A}$ $\otimes$
$\mathcal{B}$ as a superalgebra with multiplication defined by
\begin{equation*}
(a\otimes b)(a^{\prime}\otimes b^{\prime})
=(-1)^{|b||a^{\prime}|}(aa^{\prime }\otimes bb^{\prime})\text{ \ \ \
\ \ \ \ }(a,a^{\prime}\in\mathcal{A},\text{
}b,b^{\prime}\in\mathcal{B})
\end{equation*}
where $|b|$ denotes the $\Z_2$-degree of $b$, etc. Also, we shall
use short-hand notation $ab$ for $(a\otimes b) \in \mathcal{A}$
$\otimes$ $\mathcal{B}$, $a = a\otimes1$, and $b=1\otimes b$.

\begin{theorem} \label{th:isofinite}
\cite{KW} We have an isomorphism of superalgebras:
$$\Phi: \Cl_n \rtimes \C W
\stackrel{\simeq}{\longrightarrow} \Cl_n \otimes \C W^-$$
which extends the identity map on $\Cl_n$ and sends $s_i \mapsto
-\sqrt{-1} \be_i t_i.$ The inverse map $\Psi$ is the extension of
the identity map on $\Cl_n$ which sends $ t_i \mapsto \sqrt{-1}
\be_i s_i.$
(In the terminology of \cite{W2}, the superalgebras $\Cl_n \rtimes
\C W$ and $\C W^-$ are Morita super-equivalent.)
\end{theorem}

\begin{remark}
Theorem~\ref{th:isofinite} was formulated and proved in \cite{KW}
for every finite Weyl group including the exceptional types, and
the type $A$ case was due to Sergeev and Yamaguchi. See \cite{KW}
for more detail.
\end{remark}

\section{Rational double affine Hecke-Clifford algebras (DaHCa)}
\label{sec:daHCa}

In this section, we introduce the rational double affine
Hecke-Clifford algebras associated to the Weyl group $W$ of type
$A,D$ and $B$, and then establish the PBW property. The type $A$
case was treated in \cite{W1}.

\subsection{The definition of the algebras $\dahc_W$}
\subsubsection{The algebra $\dahc_W$ of type $A_{n-1}$}


We will identify $\C[\h^*] \cong \C[x_1,\ldots,x_n]$ and $\C[\h]
\cong \C[y_1,\ldots,y_n]$, where the $x_i$'s and $y_i$'s
correspond to the standard orthonormal basis $\{e_i\}$ for $\h^*$
and its dual basis $\{e_i^*\}$ for $\h$.

The following algebra $\dahc_{A_{n-1}}$ was introduced in
\cite{W1} under the notation $\mathcal A_u$. We recall it for
convenience and usage in the subsequent subsections. For $x,y$ in
an algebra $A$, we denote as usual that
$$[x,y] =xy-yx \in A.
$$
\begin{definition}
Let $u\in \C$ and $W =W_{A_{n-1}} \equiv S_n$. The rational
double affine Hecke-Clifford algebra (DaHCa) of type $A_{n-1}$,
denoted by $\dahc_W$ or $\dahc_{A_{n-1}}$, is the algebra
generated by $x_i, y_i, c_i$, $1\le i \le n$ and $W$, subject to
the relation (\ref{clifford}) among $c_i$'s and the following
relations (where we identify $\h^* =\C x_1+\cdots+\C x_n$ and $\h
=\C y_1+\cdots+\C y_n$):
\begin{align}
%
%
%
 x_{i}x_{j}& =x_{j}x_{i}, \; y_{i}y_{j}=y_{j}y_{i}, \;
 y_i c_j =c_j y_i \quad (\forall i\; \forall j) \label{comm} \\
x_{i}c_{i} & =-c_{i}x_{i},\text{ }x_{i}c_{j}=c_{j}x_{i}
\quad (i\neq j) \label{commCl} \\
wxw^{-1} & =w(x)\quad (\forall x\in\h^* \quad\forall w\in W) \label{conj-x}\\
wyw^{-1} & =w(y)\quad(\forall y\in\h \quad\forall
w\in W) \label{conj-y} \\
wcw^{-1} & =w(c)\quad (\forall c\in\mathcal{C}_{n}\quad\forall
w\in W) \label{conj-c} \\
%
%
 \nonumber \\
\lbrack y_{j},x_{i}] & =u (1+c_{j}c_{i})s_{ji} \quad (i\neq j) \label{Ayjxi} \\
\lbrack y_{i},x_{i}] & = -u\sum_{k\neq i} (1+c_{k}c_{i})s_{ki}.
\label{Ayixi}
\end{align}
\end{definition}
Alternatively, we may view $u$ as a formal variable and $\dahc_W$
as a $\C(u)$-algebra. Similar remarks apply to all DaHCa, sDaHa,
and cDaHa introduced in this paper.

\subsubsection{The algebra $\dahc_W$ of type $D_{n}$}

Let $W =W_{D_n}$. Regarding elements in $W$ as even signed
permutations of $1,2,\ldots,n$ as usual, we identify the
generators $s_i \in W$, $1\leq i\leq n-1$, with transposition
$(i,i+1)$, and $s_n \in W$ with the transposition of $(n-1,n)$
coupled with the sign changes at $n-1,n$. For $1\le i\neq j\le n$,
we denote by $s_{ij} \equiv (i,j) \in W$ the transposition of $i$
and $j$, and $\overline{s}_{i j} \equiv \overline{(i,j)} \in W$
the transposition of $i$ and $j$ coupled with the sign changes at
$i,j$. By convention, we have
$$
\overline{s}_{n-1,n}\equiv \overline{(n-1,n)}=s_n,\quad
\overline{s}_{ij}\equiv \overline{(i,j)}
=s_{jn}s_{i,n-1}s_{n}s_{i,n-1}s_{jn}.
$$

\begin{definition}
Let $u\in\C$ and $W =W_{D_n}$. The rational double affine
Hecke-Clifford algebra of type $D_{n}$, denoted by $\dahc_W$ or
$\dahc_{D_n}$, is the algebra generated by $x_i, y_i, c_i$, $1\le i
\le n$ and $W$, subject to the relation (\ref{clifford}) among
$c_i$'s, (\ref{comm}--\ref{conj-c}) with the current $W$, and
(\ref{Dyjxi}--\ref{Dyixi}) below:
\begin{align}
%
\lbrack y_{j},x_{i}] & =u\left( (1+c_{j}c_{i})s_{ij}-(1-c_{j}
c_{i})\overline{s}_{ij}\right) \quad (i\neq j) \label{Dyjxi} \\
\lbrack y_{i},x_{i}] & = -u\sum_{k\neq i}\big (
(1+c_{k}c_{i})s_{ki}+(1-c_{k} c_{i})\overline{s}_{ki} \big ).
\label{Dyixi}
\end{align}
\end{definition}


\subsubsection{The algebra $\dahc_W$ of type $B_{n}$}

Let $W =W_{B_n}$. We identify $W$ as usual with the signed
permutations on $1, \ldots, n$. Regarding $W_{D_n}$ as a subgroup of
$W$, we have $s_{ij}, \overline{s}_{ij} \in W$ for $1 \le i\neq j
\le n$. Further denote $\tau_{i} \equiv \overline{(i)} \in W$ the
sign change at $i$ for $1\le i \le n$. By definition, we have
$$
\tau_n \equiv \overline{(n)} =s_n, \quad \tau_{i} \equiv
\overline{(i)} =s_{in}s_{n}s_{in}.
$$

\begin{definition}
Let $u,v\in\C$, and $W = W_{B_n}$. The rational double affine
Hecke-Clifford algebra of type $B_{n}$, denoted by $\dahc_W$ or
$\dahc_{B_n}$, is the algebra generated by $x_i, y_i, c_i$, $1\le
i \le n$ and $W$, subject to the relations (\ref{clifford}) for
$c_i$'s, (\ref{comm}--\ref{conj-c}) with the current $W$, and
(\ref{Byjxi}--\ref{Byixi}) below:
\begin{align}
%
\lbrack y_{j},x_{i}] & =u\big( (1+c_{j}c_{i})s_{ij}-(1-c_{j}
c_{i})\overline{s}_{ij}\big) \quad (i\neq j) \label{Byjxi} \\
\lbrack y_{i},x_{i}] & = -u\sum_{k\neq i}
 \big((1+c_{k}c_{i})s_{ki}+(1-c_{k}%
c_{i})\overline{s}_{ki} \big)- \sqrt{2} v\tau_{i}. \label{Byixi}
\end{align}
\end{definition}
When it is necessary to indicate the dependence of the algebra
$\dahc_W$ on $u$ and $v$, we will write $\dahc_W(u,v)$ for
$\dahc_W$.

\begin{remark}
The factor $\sqrt{2}$ in (\ref{Byixi}) is inserted to make the
definition of $\dahc_{B_n}$ compatible with the notion of sDaHa
$\sdaha_{B_n}$ below under a Morita super-equivalence $\Phi$ (cf.
Theorem~\ref{th:isomDBdaha}).
\end{remark}

\subsection{The PBW basis for $\dahc_W$}

For any classical Weyl group $W$, the algebra $\dahc_W$ is a
superalgebra by letting elements of $W$ and $x_i, y_i$ for all $i$
be even, and each $c_i$ be odd.


\begin{theorem} \label{PBW:DBdaha}
Let $W$ be $W_{A_{n-1}}$, $W_{D_n}$ or $W_{B_n}$. The multiplication
of the subalgebras $\C[\h^*], \C[\h]$, $\Cl_{n}$, and $\C W$ induces
a vector space isomorphism
\[
\C[\h^*] \otimes\Cl_n \otimes\C W \otimes \C[\h]
\stackrel{\simeq}{\longrightarrow} \dahc_W.
\]
Equivalently, the elements $\{x^\al c^\epsilon w y^\gamma| \al,
\gamma \in\Z_{+}^n, \epsilon \in\Z_2^n, w\in W\}$ form a linear
basis for $\dahc_W$ (the PBW basis).
\end{theorem}

\begin{proof}
Recall that $W$ acts diagonally on $V = \h^{*} \oplus\h$. The
strategy of proving the theorem is similar as \cite[Proof of
Th.~1.3]{EG} with one crucial modification as first observed in
\cite{W1}.

Clearly $K := \Cl_n \rtimes \C W$ is a semisimple algebra. Observe
that $E := V \otimes_{\C} K$ is a natural $K$-bimodule (even
though $V$ is not) with the right $K$-module structure on $E$
given by right multiplication and the left $K$-module structure on
$E$ by letting
\begin{eqnarray*}
%
w. (v\otimes a) &=& v^{w}\otimes wa\\
c_i. (x_j \otimes a) &=& (-1)^{\delta_{i j}} x_j \otimes (c_i
a) \\
c_i. (y_j \otimes a) &=& y_j \otimes (c_i a).
\end{eqnarray*}
where $v\in V$, $w\in W$, $a\in K$.

The rest of the proof can proceed in the same way as in
\cite[Proof of Th.~1.3]{EG}. It boils down to the verifications of
Lemmas~\ref{conj-inv-c}, \ref{conj-inv-W} and \ref{Jacobi} below,
on the conjugation invariance (by $c_i$ and $W$) of the defining
relations (\ref{Ayjxi}--\ref{Ayixi}), (\ref{Dyjxi}--\ref{Dyixi}),
or (\ref{Byjxi}--\ref{Byixi}) for type $A, D$ or $B$ respectively,
and on the Jacobi identities among the generators $x_i$'s and
$y_i$'s.
\end{proof}

\begin{remark}
Note that $\Cl_n \rtimes \C W$ is actually a subalgebra of $\dahc_W$
and the tensor product in the above theorem indicates that $\dahc_W$
has the structure of an algebra with triangular decomposition:
$$
\dahc_W \cong \C[\h^*] \otimes (\Cl_n \rtimes\C W) \otimes \C[\h].
$$
\end{remark}

The verifications of Lemmas~\ref{conj-inv-c}, \ref{conj-inv-W} and
\ref{Jacobi} below are postponed to the Appendix.
\begin{lemma} \label{conj-inv-c}
Let $W=W_{A_{n-1}}, W_{D_n}$ or $W_{B_n}$. Then the relations
(\ref{Ayjxi}--\ref{Ayixi}), (\ref{Dyjxi}--\ref{Dyixi}), or
(\ref{Byjxi}--\ref{Byixi}) are invariant under the conjugation by
$c_i$, $1\le i \le n$.
\end{lemma}

\begin{lemma} \label{conj-inv-W}
The relations (\ref{Ayjxi}--\ref{Ayixi}),
(\ref{Dyjxi}--\ref{Dyixi}), or (\ref{Byjxi}--\ref{Byixi})
respectively are invariant under the conjugation by elements in
$W_{A_{n-1}}, W_{D_n}$ or $W_{B_n}$ respectively.
\end{lemma}

\begin{lemma} \label{Jacobi}
Let $W=W_{A_{n-1}}, W_{D_n}$ or $W_{B_n}$. Then the Jacobi
identity holds for any triple among $x_i, y_i$, $1\le i \le n$ in
$\dahc_W$.
\end{lemma}

%
%

\section{The Dunkl Operators for DaHCa}
\label{sec:daHCaDunkl}

\subsection{The Dunkl representations}

The algebra $\dahc_W$ is a superalgebra by letting elements of $W$
and $x_i, y_i$ for all $i$ be even, and each $c_i$ be odd. Recall
that $\dahc_W$ admits the triangular decomposition:
$$
\dahc_W \cong \C[\h^*] \otimes K \otimes \C[\h]
$$
where we have denoted
$$
K = \Cl_n \rtimes \C W.
$$
In contrast to the usual DaHa $\daha_W$, the DaHCa $\dahc_W$ has
no automorphism which switches the subalgebras $\C [\mathfrak h]$
and $\C [\mathfrak h^*]$. Denote by $\aH_x$ and $\aH_y$ the
subalgebras of $\dahc_W$ generated by $K$ and $x_1,\ldots,x_n$,
and generated by $K$ and $y_1,\ldots,y_n$ respectively.

A $K$-module $M$ can be extended to either $\aH_x$-module
or $\aH_y$-module by demanding the action of $x_i$'s and $y_i$'s
to be trivial respectively. We define
$$
M_x:=\text{Ind}_{\aH_x}^{\dahc_W} M, \quad M_y :=
\text{Ind}_{\aH_y}^{\dahc_W} M.
$$
Below we will always use the following identification of vector
spaces:
\begin{eqnarray*}
M_x &=& \C[y_1,\ldots,y_n]\otimes M, \\
M_y &=& \C[x_1,\ldots,x_n]\otimes M.
\end{eqnarray*}
Then the action of $\dahc_W$ on $M_x$ (resp. $M_y$) is transferred
to $\C[y_1,\ldots,y_n]\otimes M$ (resp. $\C[x_1,\ldots,x_n]\otimes
M$) as follows. On $\C[y_1,\ldots,y_n]\otimes M$, $K$ acts
diagonally. More explicitly, K acts on $\C[x_1,\ldots,x_n] \otimes M$ by
\begin{align*}
w. (x_j\otimes m) &= x_j^{w}\otimes w m\\
c_i. (x_j \otimes m) &= (-1)^{\delta_{i j}} x_j \otimes c_i m
\end{align*} where $c_i \in \Cl_n$, $w \in W$. Moreover, $y_i$ acts by left multiplication in the first tensor factor, and the action of $x_i$ is given by the so-called Dunkl operators (which are generalizations of \cite{Dun}). Similarly, on $\C[x_1,\ldots,x_n]\otimes M$, $x_i$ acts by left multiplication, and $y_i$ acts by another version of Dunkl operators. In the remainder of this section we shall describe these Dunkl operators
explicitly.

\begin{remark}
A canonical choice for a $K$-module is $\Cl_n$, whose
$K$-module structure is defined by letting $\Cl_n$ act by left
multiplication and $W$ act as usual (cf. Section~\ref{subsec:iso}).
\end{remark}

\subsection{The Dunkl Operators for $\dahc_{A_{n-1}}$}

We first prepare a few lemmas. We shall denote the action of
$\sigma\in W$ on $\C[\h]$ and $\C[\h^*]$ by $f \mapsto f^\sigma$.

\begin{lemma} \label{Acomm:[y,x^l]}
Let $W=W_{A_{n-1}}$. Then the following holds in $\dahc_W$ for $l \in \Z_+$ and $i\neq j$:
\begin{eqnarray*}
\lbrack y_{i},x_{j}^l] &=& u \left( \frac{x_j^l - x_i^l}{x_j - x_i}
+
\frac{x_j^l - (-x_i)^l}{x_j + x_i}c_i c_j\right)s_{i j}.\\
\lbrack y_{i},x_{i}^l] &=& -u \sum_{k\neq i} \left( \frac{x_i^l -
x_k^l}{x_i - x_k} + \frac{x_i^l - (-x_k)^l}{x_i + x_k}c_k
c_i\right)s_{k i}.
\end{eqnarray*}
It is understood here and in similar ratios of operators below
that $\frac{h}{g} =\frac1{g} \cdot h$.
\end{lemma}
\begin{proof}
This lemma is a type $A$ counterpart of Lemma~\ref{Bcomm:[y,x^l]}
for type $B$ below. A proof can be simply obtained by modifying
the proof of Lemma~\ref{Bcomm:[y,x^l]} with the removal of those
terms involving $\overline{s}_{ij}, \overline{s}_{ki}, \tau_i$
therein. We skip the details.
\end{proof}

\begin{lemma} \label{Acomm:[y,f]}
Let $W=W_{A_{n-1}}$, and $f \in \C[x_1,\ldots,x_n]$. Then the
following identity holds in $\dahc_W$:
\begin{eqnarray*}
\lbrack y_{i},f] &=& -u \sum_{k\neq i} \left( \frac{f -
f^{s_{ki}}}{x_i - x_k} + \frac{f c_k c_i - c_k c_i f^{s_{k i}}}{x_i
+ x_k}\right)s_{k i} .
\end{eqnarray*}
\end{lemma}
\begin{proof}
It suffices to check the formula for every monomial $f$ of the
form $x_1^{l_1}\cdots x_n^{l_n}$, which follows by
Lemma~\ref{Acomm:[y,x^l]} and an induction on $a$ based on the
identity
 $$
 \lbrack y_{i},x_1^{l_1} \cdots x_a^{l_a}x_{a+1}^{l_{a+1}}]
 = \lbrack y_{i},x_1^{l_1} \cdots x_a^{l_a}]x_{a+1}^{l_{a+1}}
 + x_1^{l_1} \cdots x_a^{l_a} \lbrack y_{i}, x_{a+1}^{l_{a+1}}].
 $$
\end{proof}

Now we are ready to compute the Dunkl operator for $y_i$'s.

\begin{theorem} 
Let $W=W_{A_{n-1}}$ and $M$ be a $K$-module. The action of $y_i$
on the $\dahc_W$-module $\C[x_1,\ldots,x_n] \otimes M$ is realized
as a Dunkl operator as follows. For any polynomial $f \in
\C[x_1,\ldots,x_n]$ and $m \in M$, we have
\begin{eqnarray*}
    y_i \circ (f \otimes m) &=& - u \sum_{k\neq i}
    \left( \frac{f - f^{s_{ki}}}{x_i - x_k} + \frac{f c_k
    c_i - c_k c_i f^{s_{k i}}}{x_i + x_k}\right) \otimes s_{k i} m.
\end{eqnarray*}
\end{theorem}
\begin{proof}
We calculate that
$$
y_i \circ (f \otimes m) = [y_i,f]\otimes m + f \otimes y_i m =
[y_i,f]\otimes m.
$$
Now the result follows from Lemma ~\ref{Acomm:[y,f]}.
\end{proof}

\begin{lemma} \label{Acomm:[y^l,x]}
Let $W=W_{A_{n-1}}$. Then the following holds in $\dahc_W$ for
$l \in \Z_+$ and $i\neq j$:
\begin{eqnarray*}
\lbrack y_{j}^l,x_{i}] &=& u \frac{y_j^l - y_i^l}{y_j -
y_i}(1+c_j c_i)s_{i j}.\\
\lbrack y_{i}^l,x_{i}] &=& -u \sum_{k\neq i} \frac{y_i^l -
y_k^l}{y_i - y_k} (1+c_k c_i)s_{k i}.
\end{eqnarray*}
\end{lemma}
\begin{proof}

This lemma is a type $A$ counterpart of Lemma~\ref{Bcomm:[y^l,x]}
for type $B$ below, with the removal of those terms involving
$\overline{s}_{ij}, \overline{s}_{ki}, \tau_i$ therein. We leave a
detailed proof to the reader.
\end{proof}

\begin{lemma} \label{Acomm:[f,x]}
Let $W=W_{A_{n-1}}$, and $f \in \C[y_1,\ldots,y_n]$. Then the
following identity holds in $\dahc_W$:
\begin{eqnarray*}
\lbrack f,x_i] &=& - u \sum_{k\neq i} \frac{f - f^{s_{ki}}}{y_i -
y_k} (1+c_k c_i)s_{k i}.
\end{eqnarray*}
\end{lemma}
\begin{proof}
It suffices to check the formula for every monomial $f$, which can
be done as for the formula in Lemma~\ref{Acomm:[y,f]} using now
Lemma~\ref{Acomm:[y^l,x]} replacing Lemma~\ref{Acomm:[y,x^l]}
therein.
\end{proof}

Now we are ready to compute the Dunkl operator for $x_i$'s.
\begin{theorem} \label{ADunkl:x}
Let $W=W_{A_{n-1}}$ and $M$ be a $K$-module. The action of $x_i$
on $\C[y_1,\ldots,y_n] \otimes M$ is realized as a Dunkl operator
as follows. For any polynomial $f \in \C[y_1,\ldots,y_n]$ and $m
\in M$, we have
\begin{eqnarray*}
    x_i \circ (f \otimes m) &=&
      u \sum_{k\neq i} \frac{f - f^{s_{ki}}}{y_i - y_k}
    \otimes (1+c_k c_i)s_{k i} m.
\end{eqnarray*}
\end{theorem}
\begin{proof}
We observe that
$$
x_i \circ (f \otimes m) = [x_i,f]\otimes m + f \otimes x_i m =
[x_i,f]\otimes m.
$$
Now the result follows by Lemma~\ref{Acomm:[f,x]}.
\end{proof}

\subsection{The Dunkl Operators for $\dahc_{B_{n}}$ }

We first prepare a few lemmas. The proofs of
Lemmas~\ref{Bcomm:[y,x^l]}, \ref{Bcomm:[y,f]}, and
\ref{Bcomm:[y^l,x]} are postponed to the Appendix.

\begin{lemma} \label{Bcomm:[y,x^l]}
Let $W=W_{B_n}$. Then the following holds in $\dahc_W$ for $l \in \Z_+$ and $i\neq
j$:
\begin{eqnarray*}
\lbrack y_{i},x_{j}^l] &=& u \left(\frac{x_j^l - x_i^l}{x_j - x_i} +
\frac{x_j^l - (-x_i)^l}{x_j + x_i}c_i c_j\right)s_{i j}\\
&& \;\; - u \left(\frac{x_j^l - (-x_i)^l}{x_j + x_i} - \frac{x_j^l
-
x_i^l}{x_j - x_i}c_i c_j\right)\overline{s}_{i j}, \\
\lbrack y_{i},x_{i}^l] &=& -u \sum_{k\neq i} \left(\frac{x_i^l -
x_k^l}{x_i - x_k} +
\frac{x_i^l - (-x_k)^l}{x_i + x_k}c_k c_i\right)s_{k i} \\
&& \;\; - u \sum_{k\neq i} \left(\frac{x_i^l - (-x_k)^l}{x_i +
x_k} -\frac{x_i^l - x_k^l}{x_i - x_k}c_k c_i\right)\overline{s}_{k
i}
 - \sqrt{2}v \frac{x_i^l - (-x_i)^l}{2x_i}\tau_i.
\end{eqnarray*}
\end{lemma}

\begin{lemma} \label{Bcomm:[y,f]}
Let $W=W_{B_n}$, and $f \in \C[x_1,\ldots,x_n]$. Then the
following holds in $\dahc_W$:
\begin{eqnarray*}
\lbrack y_{i},f] &=& -u \sum_{k\neq i} \left(\frac{f -
f^{s_{ki}}}{x_i - x_k} +
\frac{f - f^{\overline{s}_{k i}}}{x_i + x_k}c_k c_i\right)s_{k i} \\
&& - u \sum_{k\neq i} \left(\frac{f - f^{\overline{s}_{k i}}}{x_i
+ x_k} - \frac{f - f^{s_{k i}}}{x_i - x_k}c_k
c_i\right)\overline{s}_{k i} - \sqrt{2}v \frac{f -
f^{\tau_i}}{2x_i} \tau_i.
\end{eqnarray*}
\end{lemma}

Now we are ready to compute the Dunkl operator for $y_i$'s.
\begin{theorem} 
Let $W=W_{B_n}$ and $M$ be a $K$-module. The action of $y_i$ on
$\C[x_1,\ldots,x_n] \otimes M$ is realized as follows. For any
polynomial $f \in \C[x_1,\ldots,x_n]$ and $m \in M$, we have
\begin{eqnarray*}
    y_i \circ (f \otimes m) &=&
  - u \sum_{k\neq i} \left(\frac{f - f^{s_{ki}}}{x_i - x_k} + \frac{f -
    f^{\overline{s}_{k i}}}{x_i + x_k}c_k c_i\right) \otimes s_{k i}
    m \\
    && - u \sum_{k\neq i} \left(\frac{f - f^{\overline{s}_{k i}}}{x_i +
    x_k} - \frac{f - f^{s_{k i}}}{x_i - x_k}c_k c_i\right)\otimes \overline{s}_{k
    i}m \\
    &&- \sqrt{2}v \frac{f - f^{\tau_i}}{2x_i} \otimes \tau_i m.
\end{eqnarray*}
\end{theorem}
\begin{proof}
We observe that
$$
y_i \circ (f \otimes m) = [y_i,f]\otimes m + f \otimes y_i m =
[y_i,f]\otimes m.
$$
Now the result follows from Lemma ~\ref{Bcomm:[y,f]}.
\end{proof}

\begin{lemma} \label{Bcomm:[y^l,x]}
Let $W=W_{B_n}$. Then the following holds in $\dahc_W$ for $l \in \Z_+$ and $i\neq
j$:
\begin{eqnarray*}
\lbrack y_{j}^l,x_{i}] &=& u \left(\frac{y_j^l - y_i^l}{y_j -
y_i}(1+c_j c_i)s_{i j} - \frac{y_j^l - (-y_i)^l}{y_j + y_i}(1-c_j
c_i)
\overline{s}_{i j}\right).\\
\lbrack y_{i}^l,x_{i}] &=& -u \sum_{k\neq i} \frac{y_i^l -
y_k^l}{y_i -
y_k} (1+c_k c_i)s_{k i}\\
&& - u \sum_{k\neq i} \frac{y_i^l - (-y_k)^l}{y_i + y_k} (1-c_k
c_i)\overline{s}_{k i} - \sqrt{2}v \frac{y_i^l -
(-y_i)^l}{2y_i}\tau_i.
\end{eqnarray*}
\end{lemma}

Similarly as before, we can derive the next lemma from
Lemma~\ref{Bcomm:[y^l,x]}.
\begin{lemma} \label{Bcomm:[f,x]}
Let $W=W_{B_n}$, and $f \in \C[y_1,\ldots,y_n]$. Then the
following identity holds in $\dahc_W$:
\begin{eqnarray*}
\lbrack f,x_i] &=&
 - u \sum_{k\neq i} \frac{f - f^{s_{ki}}}{y_i - y_k}
    (1+c_k c_i)s_{k i}\\
    && - u \sum_{k\neq i} \frac{f - f^{\overline{s}_{k i}}}
    {y_i + y_k} (1-c_k c_i)\overline{s}_{k i}
    - \sqrt{2}v \frac{f - f^{\tau_i}}{2y_i}\tau_i.
\end{eqnarray*}
\end{lemma}

Now we are ready to compute the Dunkl operator for $x_i$'s.

\begin{theorem} \label{BDunkl:x}
Let $W=W_{B_n}$. The action of $x_i$ on $\C[y_1,\ldots,y_n]
\otimes M$ is realized as follows. For any polynomial $f \in
\C[y_1,\ldots,y_n]$ and $m \in M$, we have
\begin{eqnarray*}
    x_i \circ (f \otimes m) &=&
    u \sum_{k\neq i} \frac{f - f^{s_{ki}}}{y_i - y_k}
    \otimes (1+c_k c_i)s_{k i} m \\
    && + u \sum_{k\neq i} \frac{f - f^{\overline{s}_{k i}}}{y_i +
    y_k} \otimes (1-c_k c_i)\overline{s}_{k i}m
    + \sqrt{2}v \frac{f - f^{\tau_i}}{2y_i} \otimes\tau_i m.
\end{eqnarray*}
\end{theorem}
\begin{proof}
We observe that
$$
x_i \circ (f \otimes m) = [x_i,f]\otimes m + f \otimes x_i m =
[x_i,f]\otimes m.
$$
Now the result follows from Lemma
    ~\ref{Bcomm:[f,x]}.
\end{proof}

\subsection{The Dunkl Operators for $\dahc_{D_{n}}$}

Below, the actions of $x_i$'s and $y_i$'s are realized as Dunkl
operators. Due to the similarity of the bracket relations $\lbrack
-,-]$ in $\dahc_{D_{n}}$ and $\dahc_{B_{n}}$ (e.g. compare the
type $D$ relation (\ref{Dyixi}) with the type $B$ relation
(\ref{Byixi})), the formulas below for type $D_n$ are obtained
from their type $B_n$ counterparts in the previous subsection by
dropping the terms involving the parameter $v$. The proofs are the
same as for the type $B$, and thus will be skipped.

\begin{lemma} \label{Dcomm:[y,f]}
Let $W=W_{D_n}$, and $f \in \C[x_1,\ldots,x_n]$. Then the
following holds in $\dahc_W$:
\begin{eqnarray*}
\lbrack y_{i},f] &=&
 -u \sum_{k\neq i} \left(\frac{f -
f^{s_{ki}}}{x_i - x_k} +
\frac{f - f^{\overline{s}_{k i}}}{x_i + x_k}c_k c_i\right)s_{k i} \\
&& \quad - u \sum_{k\neq i} \left(\frac{f - f^{\overline{s}_{k
i}}}{x_i + x_k} - \frac{f - f^{s_{k i}}}{x_i - x_k}c_k
c_i\right)\overline{s}_{k i}.
\end{eqnarray*}
\end{lemma}

\begin{lemma} \label{Dcomm:[f,x]}
Let $W=W_{D_n}$, and $f \in \C[y_1,\ldots,y_n]$. Then the
following identity holds in $\dahc_W$:
\begin{eqnarray*}
\lbrack f,x_i] &=&
 - u \sum_{k\neq i} \frac{f - f^{s_{ki}}}{y_i - y_k}
    (1+c_k c_i)s_{k i}
     - u \sum_{k\neq i} \frac{f - f^{\overline{s}_{k i}}}
    {y_i + y_k} (1-c_k c_i)\overline{s}_{k i}.
\end{eqnarray*}
\end{lemma}

\begin{theorem} 
Let $W=W_{D_n}$, and let $M$ be a $K$-module. The action of $y_i$
on $\C[x_1,\ldots,x_n] \otimes M$ is realized as a Dunkl operator
as follows. For any polynomial $f \in \C[x_1,\ldots,x_n]$ and $m
\in M$, we have
\begin{eqnarray*}
    y_i \circ (f \otimes m) &=&
     - u \sum_{k\neq i} \left( \frac{f - f^{s_{ki}}}{x_i - x_k} + \frac{f -
    f^{\overline{s}_{k i}}}{x_i + x_k}c_k c_i\right) \otimes s_{k i}
    m \\
    && \quad - u \sum_{k\neq i} \left(\frac{f - f^{\overline{s}_{k i}}}{x_i +
    x_k} - \frac{f - f^{s_{k i}}}{x_i - x_k}c_k c_i\right)\otimes \overline{s}_{k
    i}m.
\end{eqnarray*}
\end{theorem}

\begin{theorem} 
Let $W=W_{D_n}$, and let $M$ be a $K$-module. . The action of
$x_i$ on $\C[y_1,\ldots,y_n] \otimes M$ is realized as follows.
For any $f \in \C[y_1,\ldots,y_n]$ and $m \in M$, we have
\begin{eqnarray*}
    x_i \circ (f \otimes m) &=&
     u \sum_{k\neq i} \frac{f - f^{s_{ki}}}{y_i - y_k}
    \otimes (1+c_k c_i)s_{k i} m \\
    && \quad + u \sum_{k\neq i} \frac{f - f^{\overline{s}_{k i}}}{y_i +
    y_k} \otimes (1-c_k c_i)\overline{s}_{k i}m .
\end{eqnarray*}
\end{theorem}
\subsection{The even center for $\dahc_W$}

Recall that the {\em even center} $\mathcal Z (A)$ of a superalgebra
$A$ consists of the even central elements of $A$. It turns out the
algebra $\dahc_W$ has a large center.

\begin{proposition} \label{CenDaHa}
Let $W$ be $W_{A_{n-1}}$, $W_{D_n}$ or $W_{B_n}$. The even center
$\mathcal Z(\dahc_W)$ contains $\C[y_1,\ldots, y_n]^W$ and $\C
[x_1^2,\ldots, x_n^2]^W$ as subalgebras. In particular, $\dahc_W$
is module-finite over its even center.
\end{proposition}
\begin{proof}
Let $f\in\C[y_1,\ldots, y_n]^W$. Then by the definition of
$\dahc_W$, $f$ commutes with $\Cl_n$, $W$, and $y_i$ for all
$1\leq i \leq n$. Since $f = f^{w}$ for all $w \in W$, it follows
by Lemmas \ref{Acomm:[f,x]}, \ref{Bcomm:[f,x]} or
\ref{Dcomm:[f,x]} (for type $A, D$ or $B$ respectively) that
$\lbrack f, x_i] = 0$ for each $i$. Hence $f$ commutes with
$\Cl_n$, $W$, and $\C[x_1,\ldots,x_n]$. Therefore $f$ is in the
even center $\mathcal Z(\dahc_W)$.

Suppose now that $f \in \C [x_1^2,\ldots, x_n^2]^W$, then by the
definition of $\dahc_W$, $f$ commutes with $\Cl_n$, $W$, and $x_i$
for all $1\leq i \leq n$. By Lemma \ref{Acomm:[y,f]},
\ref{Bcomm:[y,f]} or \ref{Dcomm:[y,f]} (for type $A, D$ or $B$
respectively), we have $\lbrack y_i, f] = 0$ for each $i$. Therefore
$f$ is in the even center.

The module-finiteness over the even center now follows from the
PBW property of $\dahc_W$ (see Theorem~\ref{PBW:DBdaha}).
\end{proof}

\section{Rational spin double affine Hecke algebras (sDaHa)}
\label{sec:sdaha}

In this section, we introduce the rational spin double affine
Hecke algebras associated to the Weyl group $W$ of type $A_{n-1},
D_n$ and $B_n$, and then establish their PBW property.
\subsection{Elements in $\C W^-$ of order $2$}
Recall that the spin group algebra $\C W^-$ has a presentation
with generator $t_i$ given in Section~\ref{sec:finite}. Introduce
the following notation
\begin{eqnarray*}
t_{i\uparrow j} &=& \left \{
 \begin{array}{ll}
 t_it_{i+1}\cdots t_j, & \text{ if } i\leq j\\
 1, &\text{ otherwise},
 \end{array}
 \right. \\
t_{i\downarrow j} &=& \left \{
 \begin{array}{ll}
 t_it_{i-1}\cdots t_j, & \text{ if } i\geq j\\
 1, &\text{ otherwise}.
 \end{array}
 \right.
\end{eqnarray*}
Define the following odd elements in $\C W^-$ of order $2$, which
are analogs of reflections in $W$, for $1\leq i\,<j\leq n$:
{\allowdisplaybreaks
\begin{eqnarray*}
 t_{i j} \equiv [i,j]
 &=& (-1)^{j-i-1}t_{j-1}\ldots t_{i+1}t_{i}t_{i+1}\ldots t_{j-1}\\
 t_{ji} \equiv {[}j,i] &= & -[i,j] \\
\overline{t}_{i j} \equiv \overline{[i,j]}
 &=& \left \{
 \begin{array}{ll}
  (-1)^{j-i-1} t_{j \uparrow n-1} t_{i \uparrow n-2}
  {t}_{n}
   t_{n-2 \downarrow i} t_{n-1 \downarrow j}, & \text{for type } D_n \\
  (-1)^{j-i} t_{j \uparrow n-1} t_{i \uparrow n-2}
   {t}_{n}{t}_{n-1}{t}_n
   t_{n-2 \downarrow i} t_{n-1 \downarrow j}, & \text{for type } B_n
  \end{array}
  \right.
\\
\overline{t}_{ji} \equiv \overline{[j,i]} &=& \overline{[i,j]} \\
 \overline{t}_i \equiv \overline{[i]}
  &=& (-1)^{n-i} t_i\cdots t_{n-1}t_n t_{n-1}\cdots t_i \qquad (1\le i \leq n).
\end{eqnarray*}
}

Note the natural inclusions of algebras $\C W_{A_{n-1}}^- \leq \C
W_{D_n}^- \leq \C W^-_{B_n}$. In particular, $t_1,\ldots, t_{n-1}$
and ${t}_{n}{t}_{n-1}{t}_n$ generate a subalgebra of $\C
W^-_{B_n}$ which is isomorphic to $\C W_{D_n}^-$ (where
$\,-{t}_{n}{t}_{n-1}{t}_n$ corresponds to the $n$-th generator for
$\C W_{D_n}^-$). Hence, the notations $[i,j], \overline{[i,j]}$
here are consistent with such a subalgebra structure. Although we
will not use it in this paper, we can show for $i<j$ that
$\overline{[i,j]} ={[}j,n] [i,n-1] {t}_{n} [i,n-1] [j,n]$.

\subsection{The algebra $\sdaha_W$ of type $A_{n-1}$}

The following algebra $\sdaha_{A_{n-1}}$ was introduced in
\cite{W1} under the notation of $\mathcal B_u$. We recall the
definition here for the convenience of the subsequent subsections.

\begin{definition}
Let $u\in \C$, and let $W =W_{A_{n-1}}$. The rational spin double
affine Hecke algebra of type $A_{n-1}$, denoted by $\sdaha_W$ or
$\sdaha_{A_{n-1}}$, is the algebra generated by $\xi_{i},y_{i}$ for
$1\leq i\leq n$ and $\C W^-$, subject to the following relations:
{\allowdisplaybreaks
\begin{align}
%
%
%
y_{i}y_{j} =y_{j}y_{i}, & \quad
%
\xi_{i}\xi_{j} =-\xi_{j}\xi_{i} \quad (i\neq j) \label{xi} \\
t_{i}y_{i} =y_{i+1}t_{i}, & \quad t_{i}\xi_{i} =-\xi_{i+1}t_{i}
 \\
t_{i}y_{j} =y_{j}t_{i}, & \quad t_{i}\xi_{j} =-\xi_{j}t_{i},\quad
(j\neq i,i+1) \label{sapart} \\
 \nonumber \\
\lbrack y_{j},\xi_{i}] & =-u[i,j] \quad \qquad (i\neq j)\\
\lbrack y_{i},\xi_{i}] & = u\sum_{k\neq i} [i,k].
\end{align}
}
\end{definition}
\subsection{The algebra $\sdaha_W$ of type $D_{n}$}

\begin{definition}
Let $u\in \C$, and let $W =W_{D_n}$. The rational spin double affine
Hecke algebra of type $D_n$, denoted by $\sdaha_W$ or
$\sdaha_{D_n}$, is the algebra generated by $\xi_{i},y_{i}$ for
$1\leq i\leq n$ and $\C W^-$, subject to the relations
(\ref{xi}--\ref{sapart}) and the following additional relations:
{\allowdisplaybreaks
\begin{align*}
t_{n}y_{n} =-y_{n-1}t_{n}, & \quad
 t_{n}\xi_{n} =-\xi_{n-1}t_{n}
 \\
t_{n}y_{j} =y_{j}t_{n}, & \quad t_{n}\xi_{j} =-\xi_{j}t_{n},\quad
\quad (j\neq n-1,n) \\
\lbrack y_{j},\xi_{i}] & =-u[i,j]+u\overline{[i,j]}\quad (i\neq j) \\
\lbrack y_{i},\xi_{i}] & = u\sum_{k\neq i}\left( [i,k]+\overline
{[i,k]}\right).
\end{align*}
}
\end{definition}

\subsection{The algebra $\sdaha_W$ of type $B_n$}

\begin{definition}
Let $u,v\in \C$, and $W=W_{B_n}$. The rational spin double affine
Hecke algebra of type $B_n$, denoted by $\sdaha_W$ or
$\sdaha_{B_n}$, is the algebra generated by $\xi_{i},y_{i}$ for
$1\leq i\leq n$ and $\C W_{B_n}^-$, subject to the relations
(\ref{xi}--\ref{sapart}) and the following additional relations:
{\allowdisplaybreaks
\begin{align*}
t_{n}y_{n}=-y_{n}t_{n}, &\quad t_{n}\xi_{n} = -\xi_{n}t_{n}
\\
t_{n}y_{j} =y_{j}t_{n}, & \quad t_{n}\xi_{j} =-\xi_{j}t_{n},\quad
\quad (j\neq n) \\
\lbrack y_{j},\xi_{i}] & =-u[i,j]+u\overline{[i,j]}\quad (i\neq j) \\
\lbrack y_{i},\xi_{i}]
 & = u\sum_{k\neq i}\left( [i,k]+\overline {[i,k]}\right) +v \overline{[i]}.
\end{align*}
}
\end{definition}
We write $\sdaha_W (u,v)$ for $\sdaha_W$ to indicate the
dependence on $u,v$ if necessary.
%
%
\subsection{Isomorphism of superalgebras}

The algebra $\sdaha_W$ contains several distinguished subalgebras:
the skew-polynomial algebra $\Cl[\xi_1,\ldots,\xi_n]$, the spin
Weyl group algebra $\C W^-$, and the polynomial algebra
$\C[y_1,\ldots,y_n]$. The algebra $\sdaha_W$ has a superalgebra
structure with $y_i$ even and $\xi_i, t_i$ odd for all $i$.

\begin{lemma} \label{identify}
Let $W$ be one of the Weyl groups $W_{A_{n-1}}$, $W_{D_n}$ or
$W_{B_n}$. The map $\Phi: \Cl_n \rtimes \C W \rightarrow \Cl_n
\otimes \C W^-$ (which is an isomorphism by
Theorem~\ref{th:isofinite}) sends
\begin{align}
(c_k-c_i) s_{ik} &\longmapsto -\sqrt{-2}\;[k,i] \label{transpositionA} \\
(c_k +c_i) \overline{s}_{ik} &\longmapsto
-\sqrt{-2}\;\overline{[k,i]}
 \label{transpositionD} \\
c_i \tau_i &\longmapsto -\sqrt{-1} \, \overline{[i]}
\label{transpositionB}
\end{align}
for $i \neq k$, whenever it is applicable.
\end{lemma}

\begin{proof}
We may assume that $i>k$ without loss of generality.

We prove (\ref{transpositionA}) by induction on $i$. First,
(\ref{transpositionA}) for $i=k+1$ holds by
Theorem~\ref{th:isofinite}. Assuming that (\ref{transpositionA})
holds for $i$, i.e. $\Phi((c_k -c_i) s_{ik}) = -\sqrt{-2}\;[k,i]$,
we have by Theorem~\ref{th:isofinite} and the definition of
${[k,i]}$ that
\begin{align*}
\Phi((c_k -c_{i+1}) s_{{i+1}, k})
 & =\Phi(s_i(c_k -c_i) s_{ik}s_i) \\
 & =(-\sqrt{-1} \be_it_i ) (-\sqrt{-2}\;[k,i]) (-\sqrt{-1} \be_it_i) \\
 & = \sqrt{-2}\; t_i [k,i] t_i
  =-\sqrt{-2} [k,i+1].
\end{align*}

We now prove (\ref{transpositionB}) by a similar downward
induction on $i$, whose initial case $i=n$ is taken care of by
Theorem~\ref{th:isofinite}. Assume that (\ref{transpositionB})
holds for $i+1 \le n$, i.e. $\Phi (c_{i+1}\tau_{i+1}) = -
\sqrt{-1} \overline{[i+1]}.$ Then, by Theorem~\ref{th:isofinite}
and the definition of $\overline{[i]}$, we have
\begin{align*}
\Phi (c_{i}\tau_{i})
 &=\Phi (s_i c_{i+1}\tau_{i+1} s_i) \\
 & =(-\sqrt{-1} \be_it_i ) ( - \sqrt{-1} \, \overline{[i+1]}) (-\sqrt{-1} \be_it_i) \\
 & =\sqrt{-1}\, t_i\, \overline{[i+1]} \, t_i
  = - \sqrt{-1}\, \overline{[i]}.
\end{align*}

Now, we prove (\ref{transpositionD}) by downward induction first
on $k$ and then on $i$, for $W=W_{D_n}$. The initial case $i=n,
k=n-1$ holds by Theorem~\ref{th:isofinite}. Then, it follows by
the induction assumption that $\Phi( (c_{k+1} +c_n)
\overline{s}_{n,k+1}) =-\sqrt{-2}\;\overline{[k+1,n]}$,
Theorem~\ref{th:isofinite} and the definition of
$\overline{[k,n]}$ that
\begin{align*}
\Phi( (c_k +c_n) \overline{s}_{nk})
 & = \Phi( s_k(c_{k+1} +c_n) \overline{s}_{n,k+1}s_k) \\
&= (-\sqrt{-1} \be_kt_k) \cdot (-\sqrt{-2}\;\overline{[k+1,n]})
\cdot (-\sqrt{-1} \be_k t_k)\\
&= \sqrt{-2}\; t_k \overline{[k+1,n]} \, t_k
 = -\sqrt{-2}\;\overline{[k,n]}.
\end{align*}
This in turn becomes the initial step when $i=n$ for proving
(\ref{transpositionD}) by downward induction on $i$ (with fixed
$k<n$). By induction assumption (\ref{transpositionD}) holds for
$i>k+1$. Then
\begin{align*}
\Phi ((c_k +c_{i-1}) \overline{s}_{{i-1}, k})
  & =\Phi(s_{i-1}(c_k +c_i) \overline{s}_{ik}s_{i-1}) \\
 & =(-\sqrt{-1} \be_{i-1}t_{i-1}) (-\sqrt{-2}\;\overline{[k,i]}) (-\sqrt{-1} \be_{i-1}t_{i-1}) \\
 & = \sqrt{-2} t_{i-1} \overline{[k,i]} t_{i-1}
  = -\sqrt{-2}\;\overline{[k,i-1]}.
\end{align*}
This completes the proof of (\ref{transpositionD}) for type $D$.

The formula (\ref{transpositionD}) for $W=W_{B_n}$ is similarly
proved by double downward inductions on $k$ and then on $i$. The
only difference from the type $D$ case is that for type $B$ we
have to check the initial case when $k=n-1$ and $i=n$, which uses
(\ref{transpositionA}) and (\ref{transpositionB}):
\begin{align*}
\Phi( (c_{n-1} +c_n) \overline{s}_{n-1,n})
 & = \Phi( \tau_n (c_{k+1} -c_n) {s}_{n-1,n} \tau_n) \\
&= (-\sqrt{-1} c_n t_n) \cdot (-\sqrt{-2}\; t_{n-1})
\cdot (-\sqrt{-1} c_n t_n)\\
&= \sqrt{-2}\; t_n t_{n-1} t_n
 = -\sqrt{-2}\;\overline{[n-1,n]}.
\end{align*}
Thus the lemma is proved.
\end{proof}

Recall the isomorphism of superalgebras $\Phi: \Cl_n \rtimes \C W
\rightarrow \Cl_n \otimes \C W^-$ and its inverse $\Psi$ given in
Theorem~\ref{th:isofinite}.
\begin{theorem} \label{th:isomDBdaha}
Let $W$ be one of the Weyl groups $W_{A_{n-1}}$, $W_{D_n}$ or
$W_{B_n}$. Then,
\begin{enumerate}
\item there exists an isomorphism of superalgebras
 $$\Phi:\dahc_W\longrightarrow\Cl_n\otimes \sdaha_W
 $$
which extends $\Phi: \Cl_n \rtimes \C W \rightarrow \Cl_n \otimes
\C W^-$ and sends
$$
y_{i}\mapsto y_{i},\;
x_{i} \mapsto\sqrt{-2}c_{i}\xi_{i},\;
 s_i \mapsto -\sqrt{-1} \be_i t_i, \;
  c_i \mapsto c_i, \quad \forall i;
$$

\item the inverse
$$\Psi: \Cl_n\otimes \sdaha_W \longrightarrow
\dahc_W
$$
extends $\Psi: \Cl_n \otimes \C W^- \rightarrow \Cl_n \rtimes \C
W$ and sends
$$
y_{i}\mapsto y_{i},\;
\xi _{i}\mapsto\displaystyle\frac{1}{\sqrt{-2}}c_{i}x_{i},\;
t_i \mapsto \sqrt{-1} \be_i s_i,\;
c_i \mapsto c_i, \quad \forall
i.
$$
\end{enumerate}
\end{theorem}
In the terminology of \cite{W2}, $\dahc_W$ and $\sdaha_W$ are
Morita super-equivalent by Theorem~\ref{th:isomDBdaha}.

\begin{proof}
Recall that $\Phi$ extends the isomorphism $\Cl_n\rtimes\C W
\stackrel{\simeq}{\longrightarrow} \Cl_n \otimes\C W^-$. Among all
the relations (\ref{comm}-\ref{Byixi}) for $\dahc_W$, it is easy
to check that (\ref{comm}-\ref{conj-c}) are preserved by $\Phi$.
So it remains to check that $\Phi$ preserves the relations
(\ref{Ayjxi}--\ref{Ayixi}), (\ref{Dyjxi}--\ref{Dyixi}), and
(\ref{Byjxi}--\ref{Byixi}) for $W=W_{A_{n-1}}, W_{D_n}$, and
$W_{B_n}$ respectively.

We shall verify in detail that $\Phi$ preserves
(\ref{Byjxi}--\ref{Byixi}) with $W=W_{B_n}$. Indeed, by
Lemma~\ref{identify}, we have for $i\neq j$ that
\begin{align*}
\Phi (\text{l.h.s. of }(\ref{Byjxi})) &= \sqrt{-2} [y_j, c_i\xi_i] \\
& =\sqrt{-2} c_i (-u[i,j]+u\overline{[i,j]}) \\
&= \Phi \left(u\left( (1+c_{j}c_{i})s_{ji}-(1-c_{j}
c_{i})\overline{s}_{ij}\right)\right ) \\
& = \Phi (\text{r.h.s. of }(\ref{Byjxi})).
\end{align*}
Also, by Lemma~\ref{identify}, we have
\begin{align*}
\Phi (\text{l.h.s. of }(\ref{Byixi})) &= \sqrt{-2} [y_i, c_i\xi_i] \\
& = \sqrt{-2} uc_i\sum_{k\neq i}
 \left( [i,k]+\overline {[i,k]}\right) +\sqrt{-2} v c_i \overline{[i]} \\
&= \Phi \left( -u
 \sum_{k\neq i}((1+c_{k}c_{i})s_{ki}+(1-c_{k}
c_{i})\overline{s}_{ki})- \sqrt{2} v\tau_{i}\right ) \\
& = \Phi (\text{r.h.s. of }(\ref{Byixi})).
\end{align*}

By dropping the terms involving $v$ in the above equations, we
verify that the relations (\ref{Dyjxi}--\ref{Dyixi}) with
$W=W_{D_n}$ are preserved by $\Phi$. By further dropping the terms
involving $\overline{[ij]}, \overline{s}_{ij}$ etc., we can also
verify (\ref{Ayjxi}--\ref{Ayixi}) with $W =W_{A_{n-1}}$.

So, the homomorphism $\Phi$ is well defined. Similarly, one shows
that $\Psi$ is a well-defined algebra homomorphism. For example, the
relation $t_{n}\xi_{n} = -\xi_{n-1}t_{n}$ in $\sdaha_W$ for
$W=W_{D_n}$ is preserved by $\Psi$, since
\begin{align*}
\Psi(t_{n}\xi_{n})
 & =\frac{\sqrt{-1}}{\sqrt{2}} (c_{n-1} +c_n)s_n \frac{1}{\sqrt{-2}}c_n
 x_n \\
 & = \frac{1}{2}(c_{n-1} +c_n) c_{n-1} x_{n-1} s_n \\
 & = \frac{1}{2}c_{n-1} x_{n-1} (-c_{n-1} -c_n) s_n
 = -\Psi (\xi_{n-1}t_{n}).
\end{align*}

On the other hand the relation $t_{n}\xi_{n} =- \xi_{n}t_{n}$ in
$\sdaha_W$ for $W=W_{B_n}$ is preserved by $\Psi$, since
\begin{align*}
\Psi(t_{n}\xi_{n})
 & = \sqrt{-1} c_n s_n \frac{1}{\sqrt{-2}}c_n x_n
 = \frac{1}{\sqrt{2}} x_{n} s_n
 = -\frac{1}{\sqrt{2}} c_n x_{n} c_n s_n
 = -\Psi (\xi_{n}t_{n}).
\end{align*}

Since $\Phi$ and $\Psi$ are inverses on generators, they are
(inverse) algebra isomorphisms.
\end{proof}

\subsection{The PBW property for $\sdaha_W$}

We have the following PBW type property for the algebra $\sdaha_W$.
\begin{theorem} \label{PBW:sdaha}
Let $W$ be one of the Weyl groups $W_{A_{n-1}}$, $W_{D_n}$ or
$W_{B_n}$. The multiplication of the subalgebras induces an
isomorphism of vector spaces
\[
\Cl[\xi_1,\ldots,\xi_n] \otimes\mathbb{C }W^- \otimes\C
[y_1,\ldots,y_n] \longrightarrow \sdaha_W.
\]
Equivalently, the set $\{\xi^{\al} \sigma y^{\gamma}\}$ forms a
basis for $\sdaha_W$, where $\sigma$ runs over a basis for $\C W^-$,
and $\al, \gamma \in \Z_+^n$.
\end{theorem}

\begin{proof}
It follows from the defining relations for $\sdaha_W$ that
$\sdaha_W$ is spanned by the elements $\xi^{\al} \sigma y^{\gamma}$
where $\sigma$ runs over a basis for $\C W^-$, and $\al, \gamma \in
\Z_+^n$. By the isomorphism $\Psi: \Cl_n\otimes \sdaha_W
\longrightarrow\dahc_W$ in Theorem~\ref{th:isomDBdaha}, we see that
the image $\Psi (\xi^{\al} \sigma y^{\gamma})$ are linearly
independent in $\dahc_W$ by the PBW property for $\dahc_W$ (see
Theorem~\ref{PBW:DBdaha}). So the elements $\xi^{\al} \sigma
y^{\gamma}$ are linearly independent in $\sdaha_W$.

Therefore, the set $\{\xi^{\al} \sigma y^{\gamma}\}$ forms a basis
for $\sdaha_W$.
\end{proof}
The tensor product in the above theorem gives a triangular
decomposition of the algebra $\sdaha_W$.

\section{The Dunkl Operators for sDaHa} \label{sec:sdahaDunkl}

Denote by $\h_{\xi}$ the subalgebra of $\sdaha_W$ generated by
$\xi_i$'s $(1\leq i \leq n)$ and $\C W^-$. For a $\C W^-$-module V,
it can be extended to a $\h_{\xi}$- modules by letting the actions
of $\xi_i$'s on V be trivial. We define

$$
    V_{\xi}:=\text{Ind}_{\h_{\xi}}^{\sdaha_W} V \cong C[y_1,\ldots,y_n]\otimes V.
$$
We will always identify $V_{\xi} = \C[y_1,\ldots,y_n]\otimes V$. On
$\C[y_1,\ldots,y_n]\otimes V$, the element $t_i \in \C W^-$ acts as
$s_i \otimes t_i$, $y_i$ acts by left multiplication, and $\xi_i$
acts as Dunkl operators, which we will describe in this section.

Under Lemma~\ref{identify} and the superalgebra isomorphism
$\Phi:\dahc_W\rightarrow\Cl_n\otimes \sdaha_W $ in
Theorem~\ref{th:isomDBdaha}, the results in this section are fairly
straightforward counterparts of those in
Section~\ref{sec:daHCaDunkl}, and we omit the proofs.

\subsection{The Dunkl Operator for $\sdaha_{A_{n-1}}$}


The following lemma is the counterpart of Lemma~\ref{Acomm:[f,x]}.
\begin{lemma} \label{Acomm:s[f,xi]}
Let $W=W_{A_{n-1}}$, and $f \in \C[y_1,\ldots,y_n]$. Then the
following identity holds in $\sdaha_W$:
\begin{eqnarray*}
\lbrack f,\xi_i] &=& - u
    \sum_{k\neq i} \frac{f - f^{s_{ki}}}{y_i - y_k}
    \lbrack k, i].
\end{eqnarray*}
\end{lemma}

The following is the counterpart of Theorem~\ref{ADunkl:x}.
\begin{proposition} \label{ADunkl:xi}
Let $W=W_{A_{n-1}}$ and $V$ be a $\C W^-$-module. The action of
$\xi_i$ on $\C[y_1,\ldots,y_n] \otimes V$ is realized as a Dunkl
operator as follows. For any polynomial $f \in \C[y_1,\ldots,y_n]$
and $v \in V$, we have
\begin{eqnarray*}
    \xi_i \circ (f \otimes v) &=&
    u \sum_{k\neq i} \frac{f - f^{s_{ki}}}{y_i - y_k}
    \otimes \lbrack k, i] v.
\end{eqnarray*}
\end{proposition}
\subsection{The Dunkl Operator for $\sdaha_{B_{n}}$}

The following lemma is the counterpart of
Lemma~\ref{Bcomm:[y^l,x]}.
\begin{lemma} \label{Bcomm:s[y^l,xi]}
Let $W=W_{B_n}$ and $l \in \Z_+$. Then we have
\begin{eqnarray*}
\lbrack y_{j}^l,\xi_{i}] &=& u \frac{y_j^l - y_i^l}{y_j -
y_i}\lbrack j, i] + u \frac{y_j^l - (-y_i)^l}{y_j + y_i}
\overline{\lbrack j, i]}.\\
\lbrack y_{i}^l,\xi_{i}] &=&
 -u \sum_{k\neq i} \frac{y_i^l - y_k^l}{y_i -
y_k} \lbrack k, i]\\
&&\quad + u \sum_{k\neq i} \frac{y_i^l - (-y_k)^l}{y_i + y_k}
\overline{\lbrack k, i]} + v \frac{y_i^l -
 (-y_i)^l}{2y_i} \overline{\lbrack i]}.
\end{eqnarray*}
\end{lemma}

The following lemma is the counterpart of Lemma~\ref{Bcomm:[f,x]}.
\begin{lemma} \label{Bcomm:s[f,xi]}
Let $W=W_{B_n}$. The following identity holds in $\sdaha_W$:
\begin{eqnarray*}
\lbrack f, \xi]
 &=& - u\sum_{k\neq i} \frac{f - f^{s_{ki}}}{y_i - y_k} \lbrack k,i]
  + u \sum_{k\neq i} \frac{f - f^{\overline{s}_{k i}}}{y_i + y_k}
\overline{\lbrack k,i]} + v \frac{f - f^{\tau_i}}{2y_i}
\overline{\lbrack i]}.
\end{eqnarray*}
\end{lemma}

The following is the counterpart of Theorem~\ref{BDunkl:x}.
\begin{proposition}
Let $W=W_{B_n}$, $V$ be a $\C W^-$-module. The action of $\xi_i$
on $\C[y_1,\ldots,y_n] \otimes V$ is realized as a Dunkl operator
as follows. For any polynomial $f \in \C[y_1,\ldots,y_n]$ and $v
\in V$, we have
\begin{eqnarray*}
    \xi_i \circ (f \otimes v) &=&
    u \sum_{k\neq i} \frac{f - f^{s_{ki}}}{y_i - y_k}
    \otimes \lbrack k,i] v \\
    && - u \sum_{k\neq i} \frac{f - f^{\overline{s}_{k i}}}{y_i +
    y_k} \otimes \overline{\lbrack k,i]}v
    - v \frac{f - f^{\tau_i}}{2y_i} \otimes\overline{\lbrack i]} v.
\end{eqnarray*}
\end{proposition}

\subsection{The Dunkl Operator for $\sdaha_{D_{n}}$}

\begin{proposition}
Let $W=W_{D_n}$, $V$ be a $\C W^-$-module. The action of $\xi_i$
on $\C[y_1,\ldots,y_n] \otimes V$ is realized as a Dunkl operator
as follows. For any polynomial $f \in \C[y_1,\ldots,y_n]$ and $v
\in V$, we have
\begin{eqnarray*}
    \xi_i \circ (f \otimes v) &=&
    u \sum_{k\neq i} \frac{f - f^{s_{ki}}}{y_i - y_k}
    \otimes \lbrack k,i] v
    -u \sum_{k\neq i} \frac{f - f^{\overline{s}_{k i}}}{y_i +
    y_k} \otimes \overline{\lbrack k,i]}v .
\end{eqnarray*}
\end{proposition}

\subsection{The even center for $\sdaha_W$}

\begin{proposition}
Let $W$ be one of the Weyl groups $W_{A_{n-1}}$, $W_{D_n}$ or
$W_{B_n}$. The even center for $\sdaha_W$ contains $\C[y_1,\ldots,
y_n]^W$ and $\C [\xi_1^2,\ldots, \xi_n^2]^W$. In particular,
$\sdaha_W$ is module-finite over its even center.
\end{proposition}

\begin{proof}

By the isomorphism $\Phi: \dahc_W \rightarrow \Cl_n \otimes
\sdaha_W$ (see Theorems~\ref{th:isomDBdaha}) and
Proposition~\ref{CenDaHa}, we have

\[
\C[y_1,\ldots, y_n]^W \subseteq \Phi(Z(\dahc_W)) = Z(\Cl_n \otimes
\sdaha_W),
\]
\[
 \C [\xi_1^2,\ldots, \xi_n^2]^W \subseteq \Phi(Z(\dahc_W)) =
Z(\Cl_n \otimes \sdaha_W).
\]
The first statement follows by noting that $\C[y_1,\ldots, y_n]^W$
and $\C[\xi_1^2,\ldots, \xi_n^2]^W$ actually lie in $\sdaha_W$. The
second statement now follows from the PBW property of $\sdaha_W$
(Theorem~\ref{PBW:sdaha}).
\end{proof}
\section{Rational covering double affine Hecke algebras (cDaHa)}
\label{sec:coveringHecke}

In this section, the rational covering double affine Hecke
algebras (cDaHa) $\cdaH_W$ associated to classical Weyl groups $W$
are introduced. It has as its natural quotients the usual rational
DaHa $\daha_W$ \cite{EG} (which will be recalled below) and the
rational sDaHa $\sdaha_W$ introduced in Section~\ref{sec:sdaha}.

\subsection{``Reflections" in $\wtd{W}$}

Recall the distinguished double cover $\wtd{W}$ of a Weyl group $W$
with generators $\td{t}_i$'s from Section~\ref{subsec:spinWeyl}.

Introduce the notation
\begin{eqnarray*}
\td{t}_{i\uparrow j} &=& \left \{
 \begin{array}{ll}
 \td{t}_i\td{t}_{i+1}\cdots \td{t}_j, & \text{ if } i\leq j\\
 1, &\text{ otherwise},
 \end{array}
 \right. \\
\td{t}_{i\downarrow j} &=& \left \{
 \begin{array}{ll}
 \td{t}_i\td{t}_{i-1}\cdots \td{t}_j, & \text{ if } i\geq j\\
 1, &\text{ otherwise}.
 \end{array}
 \right.
\end{eqnarray*}

Define the following elements in $\wtd{W}$, which are
distinguished preimages of reflections in $W$ under the canonical
map $\wtd{W} \rightarrow W$, for $1\leq i\,<j\leq n$:
{\allowdisplaybreaks
\begin{eqnarray*}
\{i,j\}
  &= & z^{j-i-1}\td{t}_{j-1}\ldots \td{t}_{i+1}\td{t}_{i}\td{t}_{i+1}\ldots \td{t}_{j-1} \\
\{j,i\} &=& z \{i,j\} \\
  \overline{\{i,j\}}
 &=& \left \{
 \begin{array}{ll}
  z^{j-i-1} \td{t}_{j \uparrow n-1} \td{t}_{i \uparrow n-2}
  \td{t}_{n}
   \td{t}_{n-2 \downarrow i} \td{t}_{n-1 \downarrow j}, & \text{for type } D_n \\
  z^{j-i} \td{t}_{j \uparrow n-1} \td{t}_{i \uparrow n-2}
   \td{t}_{n} \td{t}_{n-1} \td{t}_n
   \td{t}_{n-2 \downarrow i} \td{t}_{n-1 \downarrow j}, & \text{for type } B_n
  \end{array}
  \right.
\\
\overline{\{j,i\}} &=& \overline{\{i,j\}} \\
 \{i\}
  &=& z^{n-i} \td{t}_i\cdots \td{t}_{n-1}\td{t}_n \td{t}_{n-1}\cdots \td{t}_i \qquad (1 \le i \le n).
\end{eqnarray*}
 }
We have $\{i,j\} \in \wtd{W}_{A_{n-1}}, \overline{\{i,j\}} \in
\wtd{W}_{D_{n}}$ for $1\leq i\,<j\leq n$, and $\{i\} \in
\wtd{W}_{B_{n}}$ for $1\leq i \leq n$, while noting that we have a
sequence of subgroups $\wtd{W}_{A_{n-1}} \leq \wtd{W}_{D_{n}} \leq
\wtd{W}_{B_{n}}$. The next lemma is straightforward from the
definitions, and it helps to explain our choices of notations
(recall $s_{ij} =(i,j)$ and $\overline{s}_{ij} =\overline{(i,j)}$).

\begin{lemma} \label{lem:quotient}
Let $W$ be $\wtd{W}_{A_{n-1}}, \wtd{W}_{D_{n}}$, or
$\wtd{W}_{B_{n}}$. The canonical quotient map $\Upsilon_+: \C
\wtd{W} \rightarrow \C W$ sends (for $i \neq j$)
$$
\{i,j\} \longmapsto (i,j), \; \overline{\{i,j\}} \longmapsto
 \overline{(i,j)}, \; \{i\} \longmapsto \tau_i,
$$
and the canonical quotient map $\Upsilon_-: \C \wtd{W} \rightarrow
\C W^-$ sends (for $i \neq j$)
$$
\{i,j\} \longmapsto [i,j], \; \overline{\{i,j\}} \longmapsto
\overline{[i,j]}, \; \{i\} \longmapsto \overline{[i]}
$$
whenever it makes sense for the given $W$.
\end{lemma}

\subsection{The rational Cherednik algebras}

Recall that $\mathfrak h =\C^n$, and we have identified $\C
[\mathfrak h] =\C [x_1,\ldots,x_n]$ and $\C [\mathfrak h^*] =\C
[y_1,\ldots, y_n]$. Below we shall recall the definition \cite{EG}
of rational double affine Hecke algebras (also called rational
Cherednik algebras) associated to the classical Weyl groups in a
more concrete form.

Let $t,u \in \C$. Let $W$ be one of the Weyl groups $W_{A_{n-1}},
W_{D_{n}}$, or $W_{B_{n}}$ respectively. The rational Cherednik
algebra $\daha_W$ is the algebra generated by $x_i,y_i$ $(1\leq i
\leq n)$ and $W$, subject to the common relations
(\ref{xx;yy}--\ref{xykappa:S_n}), and the additional relations
(\ref{commAij}--\ref{commAii}) for type $A$,
(\ref{commDij}--\ref{commDii}) for type $D$,
(\ref{commBij}--\ref{commBii}) for type $B$, respectively:
\begin{align}
x_i x_j = x_j x_i, & \quad y_i y_j = y_j y_i \qquad (\forall
i,j) \label{xx;yy} \\
\sigma x = x^\sigma \sigma, & \quad \sigma y = y^{\sigma}\sigma
\qquad (\sigma \in W, x\in \mathfrak h^*, y\in \mathfrak h) \label{xykappa:S_n}\\
\lbrack y_j, x_i] &= u s_{ij} \qquad (i\neq j) \label{commAij}\\
\lbrack y_i, x_i] &= t\cdot 1 -u \sum_{k\neq i} s_{k i}
\label{commAii} \\
\lbrack y_j, x_i] &= u(s_{i j}-\overline{s}_{i j}) \qquad (i\neq j) \label{commDij}\\
\lbrack y_i, x_i] &= t \cdot 1 -u \sum_{k\neq i}(s_{k
i}+\overline{s}_{k i}) \label{commDii}\\
\lbrack y_j, x_i] &= u(s_{i j}-\overline{s}_{i j})
 \qquad (i\neq j) \label{commBij}\\
\lbrack y_i, x_i] &= t \cdot 1 -u \sum_{k\neq i}(s_{k
i}+\overline{s}_{ki}) - v\tau_i. \label{commBii}
\end{align}
The algebra $\daha_W$ has the following well-known PBW property:
the multiplication of the subalgebras induces a vector space
isomorphism
\[
\C[\h^*] \otimes\C W \otimes \C[\h]
\stackrel{\simeq}{\longrightarrow} \daha_W.
\]
Equivalently, $\{x^\al w y^\gamma| \al, \gamma \in\Z_{+}^n, w\in
W\}$ form a PBW basis for $\daha_W$.

\subsection{The rational covering double affine Hecke algebra $\cdaH_W$ }

Recall that the group $\widetilde{W}$ from
Section~\ref{sec:finite} has the defining relations given in
TABLE~1, Section~\ref{sec:finite}, and $\widetilde{W}$ contains a
central element $z$ of order $2$.

\begin{definition}
Let $W = W_{A_{n-1}}$, and let $t,u\in \C$. The rational covering
double affine Hecke algebra of type $A_{n-1}$, denoted by
$\cdaH_W$ or $\cdaH_{A_{n-1}}$, is the algebra generated by
$\td{x}_i,\td{y}_i$ $(1\le i \le n)$ and $z, \td{t}_1, \ldots,
\td{t}_{n-1}$ subject to the relations for $\wtd{W}$, and the following relations: $z$ is central and
\begin{align}
%
   \td{x}_i \td{x}_j = z \td{x}_j \td{x}_i,
    &\quad%
    \td{y}_i \td{y}_j = \td{y}_j \td{y}_i \quad (i\neq j) \label{xy} \\
    \td{t}_i \td{x}_j = z \td{x}_j \td{t}_i,
    & \quad
    \td{t}_i \td{y}_j = \td{y}_j \td{t}_i\quad (j\neq i, i+1) \label{txjyj}\\
    \td{t}_i \td{x}_{i+1} = z \td{x}_i \td{t}_i ,
    &\quad
    \td{t}_i \td{y}_{i+1} = \td{y}_i \td{t}_i\label{txiyi}\\
    \lbrack \td{y}_{j},\td{x}_{i}] &= uz \{i, j\}\quad (j \neq i) \nonumber \\
    \lbrack \td{y}_{i},\td{x}_{i}] &=
    -uz \sum_{k\neq i}\{i,k\}\nonumber. 
\end{align}
\end{definition}

\begin{definition}
Let $W = W_{D_{n}}$, and let $u\in \C$. The rational covering
double affine Hecke algebra of type $D_{n}$, denoted by $\cdaH_W$
or $\cdaH_{D_{n}}$, is the algebra generated by
$\td{x}_i,\td{y}_i$ $(1\le i \le n)$ and $z, \td{t}_1, \ldots,
\td{t}_{n}$, subject to the relations for $\wtd{W}$, relations (\ref{xy}--\ref{txiyi}), and the following additional relations: $z$ is central and
\begin{align*}
    \td{t}_n \td{x}_{j} & = z \td{x}_j \td{t}_n, \qquad\quad\;\,
    \td{t}_n \td{y}_{j} = \td{y}_j \td{t}_n \quad (i\neq n-1,n) \\
    \td{t}_n \td{x}_{n} & = - \td{x}_{n-1} \td{t}_n, \qquad%
    \td{t}_n \td{y}_{n} = - \td{y}_{n-1} \td{t}_n\\
\lbrack \td{y}_{j},\td{x}_{i}]
&= uz \left (\{i,j\} - \overline{\{i,j\}} \right) \quad (j\neq i) \\
\lbrack \td{y}_{i},\td{x}_{i}] &= -uz \sum_{k\neq i} \left
(\{i,k\} + \overline{\{i,k\}} \right).
\end{align*}
\end{definition}

\begin{definition}
Let $W = W_{B_{n}}$, and let $u,v\in \C$. The rational covering
double affine Hecke algebra of type $B_{n}$, denoted by $\cdaH_W$
or $\cdaH_{B_{n}}$, is the algebra generated by
$\td{x}_i,\td{y}_i$ $(1\le i \le n)$ and $z, \td{t}_1, \ldots,
\td{t}_{n}$, subject to the relations for $\wtd{W}$, relations (\ref{xy}--\ref{txiyi}), and the following additional relations: $z$ is central and
\begin{align*}
    \td{t}_n \td{x}_{i} & = z \td{x}_i \td{t}_n, \quad\quad
    \td{t}_n \td{y}_{i} = \td{y}_i \td{t}_n \quad (i\neq n) \\
    \td{t}_n \td{x}_{n} & = - \td{x}_{n} \td{t}_n, \quad\;
    \td{t}_n \td{y}_{n} = - \td{y}_{n} \td{t}_n\\
\lbrack \td{y}_{j},\td{x}_{i}]
&= uz \left (\{i,j\} - \overline{\{i,j\}} \right) \quad (j\neq i) \\
\lbrack \td{y}_{i},\td{x}_{i}] &= -uz \sum_{k\neq i}\left
(\{i,k\} + \overline{\{i,k\}} \right) - vz {\{i\}}.
\end{align*}
\end{definition}

\subsection{PBW basis for $\cdaH_W$}

The following result provides a link between the rational
Cherednik algebra $\daha_W^{t=0}$ with the specialization $t=0$
and the rational sDaHa via the notion of rational cDaHa.

\begin{proposition} \label{dquotient}
Let $W = W_{A_{n-1}}, W_{D_n}$, or $W_{B_n}$. Then the quotient of
the rational cDaHa $\cdaH_W$ by the ideal $\langle z-1 \rangle$
(respectively, by the ideal $\langle z+1 \rangle$) is isomorphic
to the rational Cherednik algebra $\daha_W^{t=0}$ (respectively,
the rational sDaHa $\sdaha_W$).
\end{proposition}

\begin{proof}
We will merely construct the isomorphisms of superalgebras
explicitly, while noting that the verification follows directly
from the definitions of the various algebras involved.

The canonical isomorphism map $\Upsilon_+: \C \wtd{W}/\langle z -
1 \rangle \rightarrow \C W$ (cf. Lemma~\ref{lem:quotient}) can be
extended to the isomorphism of superalgebras
\begin{align*}
\Upsilon_+ : &\, \cdaH_W/ \langle z-1\rangle \longrightarrow \daha_W^{t=0}, \\
& \td{t}_i \longmapsto s_i, \; \td{x}_i \longmapsto x_i, \;
\td{y}_i\longmapsto y_i.
\end{align*}
Also, the canonical isomorphism map $\Upsilon_-: \C
\wtd{W}/\langle z + 1 \rangle \rightarrow \C W^-$ (cf.
Lemma~\ref{lem:quotient}) can be extended to the isomorphism of
superalgebras $\Upsilon_+: \cdaH_W/ \langle z+1\rangle \rightarrow
\sdaha_W$ by sending $ \td{t}_i \mapsto t_i, \; \td{x}_i \mapsto
\xi_i, \; \td{y}_i\mapsto y_i. $
\end{proof}

The next theorem follows from Proposition~\ref{dquotient}, the PBW
basis Theorem~\ref{PBW:sdaha} for $\sdaha_W$, and the PBW property
for $\daha_W$ (cf. \cite{EG}), by the same type of argument for
\cite[Proposition~3.10]{W2} or \cite[Theorem ~5.5]{KW}.
\begin{theorem}
Let $W = W_{A_{n-1}}, W_{D_n}$, or $W_{B_n}$. Then the elements
$\td{x}^{\al} \td{w}\td{y}^{\gamma}$, where $\al,\gamma \in
\Z_{+}^n \text{ \ and }\td{w} \in \wtd{W}$, form a basis for
$\cdaH_W$.
\end{theorem}

\section{Appendix: proofs of several lemmas}
\label{sec:Appendix}

\subsection{Proof of Lemma~\ref{conj-inv-c}}
We will show that the relations (\ref{Byjxi}) and (\ref{Byixi})
are invariant under the conjugation by elements $c_l$, $1\le l \le
n$. The verifications for the invariants of other relations under
the conjugation by $c_l$ are similar and will be omitted.

Consider the relation (\ref{Byjxi}) first. Clearly, (\ref{Byjxi})
is invariant under the conjugation by $c_l, l\neq i,j$. Moreover,
we calculate that
\begin{eqnarray*}
c_i \text{(r.h.s. of (\ref{Byjxi}))} c_i &=& u\big((c_i c_j
-1)s_{ji}-(-c_ic_j-1)\overline{s}_{ij}\big) \\
&=& -\lbrack y_{j}, x_{i}]
 = c_i \text{(l.h.s. of (\ref{Byjxi}))} c_i, \\
 \\
c_j \text{(r.h.s. of (\ref{Byjxi}))} c_j
&=& u\big((c_j c_i+1)s_{ji}-(-c_{j} c_{i}+1)\overline{s}_{ij}\big)\\
&=& \lbrack y_{j},x_{i}]
 = c_j \text{(l.h.s. of (\ref{Byjxi}))} c_j.
\end{eqnarray*}
Thus, (\ref{Byjxi}) is conjugation-invariant by all $c_l$.

Next, we will show that the relation (\ref{Byixi}) is invariant
under the conjugation by each $c_l$. Indeed, we have
\begin{align*}
c_i \text{(r.h.s. of } &(\ref{Byixi})) c_i \\
&= - \sqrt{2} v c_i\tau_{i} c_i
 -u\sum_{k\neq
i}c_i((1+c_{k}c_{i})s_{ki} +(1-c_{k}
c_{i})\overline{s}_{ki})c_i \\
&= \sqrt{2} v\tau_{i} -u\sum_{k\neq i}((c_i c_k -
1)s_{ki}+(-c_i c_k -1)\overline{s}_{ki})\\
&= \sqrt{2} v \tau_{i} + u\sum_{k\neq i}((1+ c_{k}c_{i})s_{ki}+(1-c_{k}%
c_{i})\overline{s}_{ki})\\
&= -\lbrack y_{i},x_{i}]
 = c_i \text{(l.h.s. of (\ref{Byixi}))}
c_i.
\end{align*}
For $j \neq i$, we have
\begin{eqnarray*}
c_j \text{(r.h.s. of (\ref{Byixi}))} c_j
&=& - \sqrt{2} v \tau_{i} -u c_j((1+c_jc_i)s_{ji}
+(1-c_{j}c_{i})\overline{s}_{ji})c_j \\
&& \qquad \quad -u\sum_{k\neq i,j}c_j((1+c_{k}c_{i})s_{ki}+(1-c_{k}%
c_{i})\overline{s}_{ki})c_j\\
&=& - \sqrt{2} v \tau_{i} -u ((c_jc_i+1)s_{ji}
+(-c_{j}c_{i}+1)\overline{s}_{ji})c_j \\
&& \qquad \quad -u\sum_{k\neq i,j}((1+c_{k}c_{i})s_{ki}+(1-c_{k}%
c_{i})\overline{s}_{ki})\\
%
%
&=& c_j \text{(l.h.s. of (\ref{Byixi}))} c_j.
\end{eqnarray*}
Therefore, the lemma is proved.

%
%
\subsection{{Proof of Lemma~\ref{conj-inv-W}}}
We will show below that the relations (\ref{Byjxi}--\ref{Byixi})
are invariant under the conjugation by elements in $W_{B_n}$. The
proof can be readily modified to yield the Weyl group invariance
of the relations (\ref{Ayjxi}--\ref{Ayixi}) and
(\ref{Dyjxi}--\ref{Dyixi}) in type $A$ and $D$ cases respectively,
and we leave the details to the interested reader.

(i) We check the invariance of (\ref{Byjxi}) under $W_{B_n}$.

Consider first the conjugation invariance by the transposition
$s_{l k}$. If $\{l, k\}\cap\{i,j\} = \emptyset$, then we have
\begin{eqnarray*}
s_{l k} \text{(r.h.s. of (\ref{Byjxi}))} s_{l k} &=& u\big((1+c_j
c_i)s_{ji}-(1-c_j c_i)\overline{s}_{ij}\big) \\
&=& \lbrack y_{j},x_{i}]
 = s_{l k} \text{(l.h.s. of (\ref{Byjxi}))} s_{l k}.
\end{eqnarray*}

If $\{l, k\}\cap\{i,j\} =\{j\}$, then we may assume $l = j$ and we
have
\begin{eqnarray*}
s_{j k} \text{(r.h.s. of (\ref{Byjxi}))} s_{j k} &=& u\big(
(1+c_{k}c_{i})s_{ik}-(1-c_{k}c_{i})\overline{s}_{ik}\big)\\
&=& \lbrack y_{k},x_{i}]
 = s_{j k} \text{(l.h.s. of (\ref{Byjxi}))} s_{j k}.
\end{eqnarray*}
We leave an entirely analogous computation when $\{l,
k\}\cap\{i,j\} =\{i\}$ to the reader.

Now, if $\{l, k\}=\{i,j\}$, then
\begin{eqnarray*}
s_{j i} \text{(r.h.s. of (\ref{Byjxi}))} s_{j i} &=& u\big(
(1+c_{i}c_{j})s_{ij}-(1-c_{i}c_{j})\overline{s}_{ij}\big) \\
&=& \lbrack y_{i},x_{j}]
 = s_{j i} \text{(l.h.s. of (\ref{Byjxi}))} s_{j i}.
\end{eqnarray*} So (\ref{Byjxi}) is invariant under the conjugation by
each transposition $s_{l k}$.

It remains to show that (\ref{Byjxi}) is invariant under the
conjugation by the simple reflection $s_n= \tau_n$. Observe that
(\ref{Byjxi}) is clearly invariant under the conjugation by $s_n$
for $n \neq j,i$ Moreover, if $j=n$ then we have
\begin{eqnarray*}
s_{n} \text{(r.h.s. of (\ref{Byjxi}))} s_{n} &=& u\big((1-c_j
c_i)\overline{s}_{ji}-(1+c_j c_i)s_{ij}\big)\\
&=& -\lbrack y_{j},x_{i}]
 = s_{n} \text{(l.h.s. of (\ref{Byjxi}))} s_{n}.
\end{eqnarray*}
If $i=n$, then we have
\begin{eqnarray*}
s_{n} \text{(r.h.s. of (\ref{Byjxi}))} s_{n} &=& u\big((1-c_j
c_i)\overline{s}_{ji}-(1+c_j c_i)s_{ij}\big)\\
&=& -\lbrack y_{j},x_{i}] = s_{n} \text{(l.h.s. of (\ref{Byjxi}))}
s_{n}.
\end{eqnarray*}
This completes (i).

 \vspace{.4cm}

(ii) We check the invariance of (\ref{Byixi}) under $W_{B_n}$.

Consider first the conjugation invariance by a transposition $s_{j
l}$. If $\{j, l\}\cap\{i\} = \emptyset$, then we have
\begin{align*}
 s_{j l} (\text{r.h.s. of } & (\ref{Byixi})) s_{j l}\\
%
%
= &
 -u s_{j l}\left((1+c_j c_i)s_{j i}+(1-c_j c_i)\overline{s}_{j
i}\right)s_{j l} \\
& -us_{j l}\left((1+c_l c_i)s_{l i}+(1-c_l c_i)\overline{s}_{l
i}\right)s_{j l}\\
& -u\sum_{k\neq i,j,l}s_{j l}((1+c_{k}c_{i})s_{ki}+(1-c_{k}%
c_{i})\overline{s}_{ki})s_{j l}- \sqrt{2} v \tau_{i}\\
= & \lbrack y_{i},x_{i}]
 =s_{j l} \text{(l.h.s. of (\ref{Byixi}))}
s_{j l}.
\end{align*}
If $\{j, l\}\cap\{i\} = \{i\}$, we may assume that $j = i$, and
then we have
\begin{align*}
s_{i l} \text{(r.h.s. of } & (\ref{Byixi})) s_{i l} \\
%
%
=&
-u s_{i l}\left((1+c_l c_i)s_{j i}+(1-c_l c_i)\overline{s}_{l
i}\right)s_{i l} \\
& -u\sum_{k\neq i,l} s_{i l}((1+c_{k}c_{l})s_{kl}+(1-c_{k}%
c_{l})\overline{s}_{kl})s_{i l} - \sqrt{2} v \tau_{l}\\
=& \lbrack y_{l},x_{l}] =s_{i l} \text{(l.h.s. of (\ref{Byixi}))}
s_{i l}.
\end{align*}

It remains to show that (\ref{Byixi}) is invariant under the
conjugation by the simple reflection $s_n \equiv \tau_n \in
W_{B_n}$. If $i \neq n$, we have
{\allowdisplaybreaks
\begin{align*}
s_{n} \text{(r.h.s. of } & (\ref{Byixi})) s_{n} \\
%
%
&= - \sqrt{2} v \tau_{i}
  -u s_n\left((1+c_n c_i)s_{n i}+(1-c_n
c_i)\overline{s}_{n i}
)\right)s_{n} \\
& -u\sum_{k\neq i,n}s_n ((1+c_{k}c_{i})s_{ki}+(1-c_{k}%
c_{i})\overline{s}_{ki})s_n \\
&= - \sqrt{2} v \tau_{i}
 -u \left((1-c_n c_i)\overline{s}_{n
i}+(1+c_n c_i)s_{n i}
)\right) \\
& -u\sum_{k\neq i,n}((1+c_{k}c_{i})s_{ki}+(1-c_{k}%
c_{i})\overline{s}_{ki})\\
&= \lbrack y_{i},x_{i}]
= s_{n} \text{(l.h.s. of (\ref{Byixi}))} s_{n}.
\end{align*}
 }
 If $i = n$, then
\begin{align*}
s_{n} \text{(r.h.s. of } & (\ref{Byixi})) s_{n} \\
&= - \sqrt{2} v \tau_{n}
-u\sum_{k\neq n}((1-c_k c_n)\overline{s}_{kn }+(1+c_k c_n)s_{kn}) \\
&= \lbrack y_{n},x_{n}]
= s_{n} \text{(l.h.s. of (\ref{Byixi}))} s_{n}.
\end{align*}
This completes the proof of (ii). Hence the lemma is proved.
%
%
\subsection{{Proof of Lemma~\ref{Jacobi}}}
We will establish the Jacobi identity for $W = W_{B_n}$. The proof
can be easily modified for the cases of type $A$ and $D$, and we
leave the details to the reader.

The Jacobi identity trivially holds among triple $x_i$'s or triple
$y_i$'s.

Now, we consider the triple with two $y$'s and one $x$. The case
with two identical $y_i$ is trivial. So we first consider $x_i$,
$y_j$, and $y_l$ where $i,j,l$ are all distinct. The Jacobi
identity holds in this case since
\begin{align*}
[x_i,[y_j,y_l]] &+[y_l,[x_i,y_j]]+[y_j,[y_l,x_i]] \\
&= 0+ [y_l, -u\big((1+c_{j}c_{i})s_{ji}-(1-c_{j}
c_{i})\overline{s}_{ij}\big)] \\
& \qquad + [y_j, u\left((1+c_{l}c_{i})s_{li}-(1-c_{l}
c_{i})\overline{s}_{il}\right)]
 =0.
\end{align*}

Now for $i\neq j$, we have
{\allowdisplaybreaks
\begin{align*}
 [x_i,[y_i,y_j& ]] +[y_j,[x_i,y_i]]+[y_i,[y_j,x_i]]\\
&= 0+ [y_j, u\sum_{k\neq i}\big ( (1+c_{k}c_{i})s_{ki}
  +(1-c_{k}c_{i})\overline{s}_{ki} \big )+ \sqrt{2} v\tau_{i}]\\
& \qquad + [y_i,u\left( (1+c_{j}c_{i})s_{ji}-(1-c_{j}
    c_{i})\overline{s}_{ij}\right)]\\
&= \Big[y_j,u\sum_{k\neq i,j}\big ( (1+c_{k}c_{i})s_{ki}
  +(1-c_{k}c_{i})\overline{s}_{ki} \big ) \Big]\\
&\qquad + [y_j,u \big ( (1+c_{j}c_{i})s_{ji}
  +(1-c_jc_{i})\overline{s}_{ji} \big )]\\
& \qquad + [y_i,u\big((1+c_{j}c_{i})s_{ji}-(1-c_{j}
    c_{i})\overline{s}_{ij}\big)]\\
&= 0+ u \big ((1+c_j c_{i})y_j s_{ji}+(1-c_j
    c_{i})y_j\overline{s}_{ji} \big) \\%
&\qquad -u \big ((1+c_jc_{i})s_{ji}y_j+(1-c_j
    c_{i})\overline{s}_{ji}y_j \big ) \\
& \qquad +u\left( (1+c_{j}c_{i})y_i s_{ji}-(1-c_{j}
    c_{i})y_i \overline{s}_{ij}\right) \\%
& \qquad -u\big((1+c_{j}c_{i})s_{ji}y_i-(1-c_{j}
    c_{i})\overline{s}_{ij}y_i\big)
    = 0.
\end{align*}
 }

Now we consider the Jacobi identity with one $y$ and two $x$'s.
The case with all distinct indices can be easily verified as
above. Moreover, for $i\neq j$, we have
{\allowdisplaybreaks
\begin{align*}
 [x_i, &[y_i,x_j ]] +[x_j,[x_i,y_i]]+[y_i,[x_j,x_i]]\\
& = [x_i, u\left( (1+c_ic_j)s_{ij}-(1-c_ic_j)\overline{s}_{ij}\right)] \\
& \quad+ \Big[x_j, u\sum_{k\neq i}\big ( (1+c_{k}c_{i})s_{ki}
  +(1-c_{k}c_{i})\overline{s}_{ki} \big )+ \sqrt{2} v\tau_{i}\Big] +0\\
& = [x_i, u\left( (1+c_ic_j)s_{ij}-(1-c_ic_j)\overline{s}_{ij}\right)] \\
&\qquad + [x_j, u \big ( (1+c_j c_{i})s_{ij}
  +(1-c_jc_{i})\overline{s}_{ij} \big ) ] \\
&= u \big ((1 -c_i c_j)x_i s_{ij} -(1 +c_ic_j) x_i\overline{s}_{ij} \big) \\
&\qquad -u \big ((1+c_ic_j)s_{ij}x_i -(1- c_ic_j)\overline{s}_{ij}x_i \big ) \\
& \qquad +u\left( (1 -c_{j}c_{i})x_j s_{ij} + (1 +c_{j}
    c_{i}) x_j \overline{s}_{ij}\right) \\
& \qquad -u\big((1+c_{j}c_{i})s_{ij}x_j +(1-c_{j}
    c_{i})\overline{s}_{ij}x_j \big)
    = 0.
\end{align*}
 }
This completes the verification of the Jacobi identity for any
triples.
%
%
\subsection{{Proof of Lemma~\ref{Bcomm:[y,x^l]}}}
We will proceed by induction. For $l=1$, then the equations hold
by (\ref{Byjxi}) and (\ref{Byixi}). Now assume that the statement
is true for $l$. Then
    {\allowdisplaybreaks
    \begin{eqnarray*}
    \lbrack y_{i},x_{j}^{l+1}] &=& \lbrack y_{i},x_{j}^{l}] x_j + x_j^l \lbrack
    y_{i},x_{j}]\\
    &=& u \left(\frac{x_j^l - x_i^l}{x_j - x_i} +
    \frac{x_j^l - (-x_i)^l}{x_j + x_i}c_i c_j\right)s_{i j}x_j\\
    && - u \left(\frac{x_j^l - (-x_i)^l}{x_j + x_i} - \frac{x_j^l -
    x_i^l}{x_j - x_i}c_i c_j\right)\overline{s}_{i j}x_j\\
    && + x_j^l u\big((1+c_{i}c_{j})s_{ij}-(1-c_{i}c_{j})\overline{s}_{ij}\big)\\
    &=& u \left(\frac{x_j^{l+1} - x_i^{l+1}}{x_j - x_i} +
    \frac{x_j^{l+1} - (-x_i)^{l+1}}{x_j + x_i}c_i c_j\right)s_{i j}\\
    && - u \left(\frac{x_j^{l+1} - (-x_i)^{l+1}}{x_j + x_i}-
    \frac{x_j^{l+1} - x_i^{l+1}}{x_j - x_i}c_i c_j\right)\overline{s}_{i j}.
    \end{eqnarray*}
    }
{\allowdisplaybreaks
\begin{eqnarray*}
    \lbrack y_{i},x_{i}^{l+1}] &=& \lbrack y_{i},x_{i}^{l}] x_i + x_i^l \lbrack
    y_{i},x_{i}]\\
    &=&
    -u \sum_{k\neq i} \left(\frac{x_i^l - x_k^l}{x_i - x_k} +
    \frac{x_i^l - (-x_k)^l}{x_i + x_k}c_k c_i\right)s_{k i}x_i \\
    && - u \sum_{k\neq i} \left(\frac{x_i^l - (-x_k)^l}{x_i + x_k} -
    \frac{x_i^l - x_k^l}{x_i - x_k}c_k c_i\right)\overline{s}_{k i}x_i \\
    && - \sqrt{2}v \frac{x_i^l \tau_i - \tau_i
    x_i^l}{2x_i} x_i\\
    && -u x_i^l\sum_{k\neq i}((1+c_{k}c_{i})s_{ki}+(1-c_{k}%
    c_{i})\overline{s}_{ki})- \sqrt{2} v x_i^l \tau_{i}\\
    &=&
    -u \sum_{k\neq i} \left(\frac{x_i^{l+1} - x_k^{l+1}}{x_i - x_k} +
    \frac{x_i^{l+1} - (-x_k)^{l+1}}{x_i + x_k}c_k c_i\right)s_{k i} \\
    && - u \sum_{k\neq i} \left(\frac{x_i^{l+1} - (-x_k)^{l+1}}{x_i + x_k} -
    \frac{x_i^{l+1} - x_k^{l+1}}{x_i - x_k}c_k c_i\right)\overline{s}_{k i}\\
    && - \sqrt{2}v \frac{x_i^{l+1} \tau_i - \tau_i x_i^{l+1}}{2x_i}.
\end{eqnarray*}
    }
This completes the proof of the identities.

%
%
\subsection{{Proof of Lemma~\ref{Bcomm:[y,f]}}}
It suffices to check the formula for every monomial $f$. First, we
consider the monomial $g = \prod_{j\neq i} x_j^{a_j}$. By
induction and Lemma \ref{Bcomm:[y,x^l]}, we can show that the
formula holds for the monomial of the form $g = \prod_{j\neq i}
x_j^{a_j}$ (the detail of the induction step does not differ much
from the following calculation). Now consider the monomial $f
=x_i^{l} g$. By Lemma~\ref{Bcomm:[y,x^l]}, we have
    {\allowdisplaybreaks
\begin{eqnarray*}
\lbrack y_i,f] &=& \lbrack y_i,x_i^{l}]g + x_i^l \lbrack y_i, g]
\\
     &=& -u \sum_{k\neq i} \left(\frac{x_i^l - x_k^l}{x_i - x_k} +
        \frac{x_i^l - (-x_k)^l}{x_i + x_k}c_k c_i\right)s_{k i}g \\
        && - u \sum_{k\neq i} \left(\frac{x_i^l - (-x_k)^l}{x_i + x_k} -
        \frac{x_i^l - x_k^l}{x_i - x_k}c_k c_i\right)\overline{s}_{k i}g \\
     && - \sqrt{2}v \frac{x_i^l - (-x_i)^l}{2x_i}\tau_i g\\
    && -u\sum_{k\neq i} x_i^l\left(\frac{g- g^{s_{ki}}}{x_i - x_k} +
    \frac{g - g^{\overline{s}_{k i}}}{x_i + x_k}c_k c_i\right)s_{k i} \\
    && - u \sum_{k\neq i} x_i^l\left(\frac{g - g^{\overline{s}_{k
    i}}}{x_i + x_k} - \frac{g - g^{s_{k i}}}{x_i - x_k}c_k
    c_i\right)\overline{s}_{k i}\\
    &=& -u \sum_{k\neq i} \left(\frac{f - f^{s_{ki}}}{x_i - x_k} +
    \frac{f - f^{\overline{s}_{k i}}}{x_i + x_k}c_k c_i\right)s_{k i} \\
    && - u \sum_{k\neq i} \left(\frac{f - f^{\overline{s}_{k i}}}{x_i +
    x_k} - \frac{f - f^{s_{k i}}}{x_i - x_k}c_k
    c_i\right)\overline{s}_{k i}
    - \sqrt{2}v \frac{f - f^{\tau_i}}{2x_i} \tau_i.
\end{eqnarray*}
 }
So the lemma is proved.

\subsection{{Proof of Lemma~\ref{Bcomm:[y^l,x]}}}
We will proceed by induction.
For $l=1$, the equations hold by
(\ref{Byjxi}) and (\ref{Byixi}). Now assume that the statement is
true for $l$. Then
\begin{eqnarray*}
    \lbrack y_{j}^{l+1},x_{i}] &=& y_j \lbrack y_{j}^{l},x_{i}] + \lbrack
    y_{j},x_{i}] y_j^l\\
    &=& u y_j \left( \frac{y_j^l - y_i^l}{y_j - y_i}(1+c_j c_i)s_{i j}
        - \frac{y_j^l - (-y_i)^l}{y_j + y_i}(1-c_j c_i)\overline{s}_{i j}\right)\\
    && + u\left( (1+c_{j}c_{i})s_{ji}-(1-c_{j}c_{i})\overline{s}_{ij}\right) y_j^l \\
    &=& u \left(\frac{y_j^{l+1} - y_i^{l+1}}{y_j -
y_i}(1+c_j c_i)s_{i j} - \frac{y_j^{l+1} - (-y_i)^{l+1}}{y_j +
y_i}(1-c_j c_i)\overline{s}_{i j}\right).
\end{eqnarray*}

On the other hand, we have
{\allowdisplaybreaks
\begin{eqnarray*}
    \lbrack y_{i}^{l+1},x_{i}] &=& y_i \lbrack y_{i}^{l},x_{i}] + \lbrack
    y_{i},x_{i}] y_i^l\\
    &=& -u \sum_{k\neq i} y_i \frac{y_i^l - y_k^l}{y_i -
    y_k} (1+c_k c_i)s_{k i}\\
    && - u \sum_{k\neq i} y_i \frac{y_i^l - (-y_k)^l}{y_i +
    y_k} (1-c_k c_i)\overline{s}_{k i} - \sqrt{2}v
    y_i \frac{y_i^l - (-y_i)^l}{2y_i}\tau_i \\
    && -u\sum_{k\neq i}((1+c_{k}c_{i})s_{ki}+(1-c_{k}%
        c_{i})\overline{s}_{ki})y_i^l- \sqrt{2} v\tau_{i}y_i^l\\
    &=& -u \sum_{k\neq i} \frac{y_i^{l+1} - y_k^{l+1}}{y_i
-y_k} (1+c_k c_i)s_{k i}\\
&& - u \sum_{k\neq i} \frac{y_i^{l+1} - (-y_k)^{l+1}}{y_i + y_k}
(1-c_k c_i)\overline{s}_{k i} - \sqrt{2}v \frac{y_i^{l+1} -
(-y_i)^{l+1}}{2y_i}\tau_i.
\end{eqnarray*}
 }
This completes the proof of the lemma.


\end{document}